\documentclass[12pt]{amsart}
\usepackage{amsmath,amsfonts,amssymb,amsthm,fullpage}

\newtheorem{proposition}{Proposition}[section]
\newtheorem{theorem}[proposition]{Theorem}
\newtheorem{corollary}[proposition]{Corollary}
\newtheorem{lemma}[proposition]{Lemma}

\theoremstyle{remark}
\newtheorem{remark}[proposition]{Remark}

\newtheorem{example}[proposition]{Example}

\numberwithin{equation}{section}

\newcommand{\nc}{\newcommand}
\nc{\I}{{\mathbf 1}}

\nc{\bN}{{\mathbf N}}
\nc{\bM}{{\mathbf M}}
\nc{\cB}{{\mathcal B}}
\nc{\cF}{{\mathcal F}}
\nc{\cM}{{\mathcal M}}
\nc{\R}{{\mathbb R}}
\nc{\N}{{\mathbb N}}
\nc{\Z}{{\mathbb Z}}
\nc{\BI}{{\mathbb I}}
\nc{\BX}{{\mathbb X}}
\nc{\BY}{{\mathbb Y}}
\nc{\cX}{{\mathcal X}}
\nc{\cH}{{\mathcal H}}
\nc{\cN}{{\mathcal N}}
\nc{\cY}{{\mathcal Y}}
\nc{\cS}{{\mathcal S}}
\nc{\cT}{{\mathcal T}}
\nc{\BH}{{\mathbb H}}
\nc{\BS}{{\mathbb S}}
\nc{\BT}{{\mathbb T}}
\nc{\bx}{\mathbf{x}}
\nc{\by}{\mathbf{y}}
\nc{\bz}{\mathbf{z}}
\nc{\bv}{\mathbf{v}}
\nc{\bw}{\mathbf{w}}
\nc{\bu}{\mathbf{u}}

\DeclareMathOperator{\card}{card}

\DeclareMathOperator{\supp}{supp}

\DeclareMathOperator{\conv}{conv}

\DeclareMathOperator{\BV}{{\mathbb Var}}
\DeclareMathOperator{\BC}{{\mathbb Cov}}
\nc{\BP}{\mathbb{P}}
\nc{\BE}{\;\mathbb{E}}
\nc{\BQ}{\mathbb{Q}}
\nc{\bH}{\overline{H}}

\renewcommand{\emptyset}{\varnothing}
\renewcommand{\phi}{\varphi}
\newcommand{\eps}{\varepsilon}
\newcommand{\deltaKS}{\boldsymbol{\delta}}

\begin{document} 

\title{Poisson hulls}


\author{G\"unter Last}
\address{G\"unter Last, Institute for Stochastics, Karlsruhe Institute of  Technology, 
  Englerstrasse 5, 76131 Karlsruhe, Germany}
\email{guenter.last@kit.edu}

\author{Ilya Molchanov}
\address{Ilya Molchanov, Institute of Mathematical Statistics and Actuarial Science, University of Bern, Alpeneggstrasse 22, 3012 Bern, Switzerland}
\email{ilya.molchanov@unibe.ch}


\begin{abstract}
  We introduce a hull operator on Poisson point processes, the easiest
  example being the convex hull of the support of a point process in
  Euclidean space. Assuming that the intensity measure of the process
  is known on the set generated by the hull operator, we discuss
  estimation of an expected linear statistic built on the Poisson
  process.  In special cases, our general scheme yields an
  estimator of the volume of a convex body or an estimator of an
  integral of a H\"older function.  We show that the estimation error
  is given by the Kabanov--Skorohod integral with respect to the
  underlying Poisson process. A crucial ingredient of our approach is
  a spatial strong Markov property of the underlying Poisson process with
  respect to the hull.  We derive the rate of normal convergence for
  the estimation error, and illustrate it on an application to
  estimators of integrals of a H\"older function.  We also discuss
  estimation of higher order symmetric statistics.
\end{abstract}



\maketitle

\section{Introduction}\label{intro}

Estimation of a convex body $K$ (a compact convex subset of
  Euclidean space), using the convex hull of points randomly sampled
from it, is a substantial area in statistical inference, see, e.g.,
\cite{brun18,cuev:fraim10}.  This convex hull is a polytope $P$, which
is a subset of $K$, and so provides biased estimators for most of
geometric parameters of $K$, e.g., its volume. While there were some
attempts to eliminate the bias by enlarging the polytope, see
\cite{rip:ras77}, only recently, \cite{BaldinReiss16} came up with an
unbiased estimator of the volume based on the observation of the
convex hull of the points from a homogenous Poisson point process
restricted to $K$. The estimator is the sum of the volume of $P$ and a
term given by the number of vertices in the convex hull normalised by
the intensity of the underlying Poisson process. A similar idea was
pursued in \cite{ReissSelk17} and \cite{ReissWahl19} in the context of
estimation of the integral of a function $\phi$ using pointwise minima
of functions that form a Poisson process with graphs lying above
$\phi$.

A common feature of these approaches is to consider certain hull
operations applied to a point process.  In the convex hull setting of
\cite{BaldinReiss16} this is the usual convex hull, and the vertices
of it are identified as points not belonging to the convex hull of
other points. In the functional setting of \cite{ReissWahl19}, the
hull is the minimum of functions and the role of vertices is played by
functions that contribute to the minimum. This paper aims to develop a
general theory for hull operators of a Poisson process
which unifies and extends the estimators considered in
\cite{BaldinReiss16,ReissSelk17,ReissWahl19}.

In the following we illustrate our general construction on the convex
hull setting of \cite{BaldinReiss16}. Let $\eta$ be a Poisson process
in $\R^d$ with a finite diffuse intensity measure $\lambda$, see
\cite{LastPenrose17}. Denote by $\conv(\eta)$ the convex hull of the
points of $\eta$ (identified with its support) and by $\partial\eta$
the set of vertices of $\conv(\eta)$.  Let $f:\R^d\to\R$ be a function
that is integrable and square integrable with respect to $\lambda$,
and define a function
\begin{align*}
  H_x(\eta):=\I\big\{x\in\big(\R^d\setminus \conv(\eta)\big)
  \cup\partial\eta\big\},\quad x\in\R^d, 
\end{align*}
which is one if $x$ does not belong to the convex hull of $\eta$ or is
a vertex.  Assume that the restriction of $\lambda$ onto $\conv(\eta)$
is known, e.g., $\lambda$ is proportional to the Lebesgue measure with
a known proportionality constant, and the aim is to estimate
$\int f(x)\,\lambda(dx)$ based on observing $\eta$. The challenge here
is to ``extrapolate'' $\lambda$ outside of the observable region given
by the convex hull of $\eta$. We show that
\begin{align*}
  \hat F:= \int \big(1-H_x(\eta)\big) f(x)\lambda(dx)
  + \int f(x) H_x(\eta-\delta_x)\,\eta(dx)
\end{align*}
is an unbiased estimator of $F:=\int f(x)\,\lambda(dx)$. Note that the
second integral on the right-hand side is the sum of $f(x)$ for all
points $x\in\eta$ such that the convex hull of $\eta$ with $x$
removed is distinct from the convex hull of $\eta$. These points are
indeed the vertices of $\conv(\eta)$ and provide an example of
the generator defined in Section~\ref{sec:hull-operator}.
The estimation error can be written as
\begin{equation}
  \label{eq:8}
  \hat F- F=\int f(x) H_x(\eta-\delta_x)\eta(dx)
  -\int f(x) H_x(\eta)\lambda(dx). 
\end{equation}
In stochastic analysis on the Poisson space, the expression on the
right-hand side is known as the Kabanov--Skorohod integral of
$f(x)H_x(\eta)$, see, e.g., \cite{Last16}. The sum over Poisson points,
as seen in the first term on the right-hand side of \eqref{eq:8}, is
well studied in stochastic geometry. Most advanced limit theorems for
such sums can be found in \cite{lac:sch:yuk19}, where these sums are
centred by subtracting the expectation. In difference to this,
\eqref{eq:8} involves subtracting a random term.

The content of the paper can be summarised as follows.  We consider
the space $\bN$ of locally finite counting measures on some space
$\BX$. In the general part of the paper (Sections 2-7) $\BX$ will be
an abstract localised Borel space.  In Section~\ref{sec:hull-operator}
we introduce and study a generator $\partial\colon\bN\to\bN$ and its
dual, the hull operator, which associates with each $\mu\in\bN$ a
measurable subset $[\mu]$ of $\BX$.  These are purely deterministic
concepts.  In Sections 3-7 we consider a Poisson process $\eta$ on
$\BX$, that is, a point process on $\BX$ (a random element of $\bN$)
with independent and Poisson distributed increments
\cite{LastPenrose17}.  In Section~\ref{secstopping} we show that,
given $\partial\eta$, the conditional distribution of $\eta$
restricted to the hull $[\eta]$ of $\eta$ is that of a Poisson process
with an appropriately trimmed intensity measure. This spatial strong
Markov property, which goes back to \cite{Zuyev99} and
\cite{osss21}, is crucial for our approach. In
Section~\ref{secestimator} we introduce Poisson hull estimators as
conditional expectations of a linear statistic of $\eta$ given
$\partial\eta$. Thanks to the spatial strong Markov property, the estimation
error turns out to be a Kabanov--Skorohod integral \cite{Last16}.  The
variance of our estimators is discussed in
Section~\ref{sec:vari-poiss-hull},  while Section~\ref{secestimatorbi}
presents an expression
for the variance of conditional expectations of higher order symmetric
statistics. The latter results also yield some new identities for
random polytopes. In Section~\ref{sec:centr-limit-theor} we discuss
the normal approximation of our estimators. Here we rely on the recent
company paper \cite{las:mol:sch20}, elaborating limit theorems for
Kabanov--Skorohod integrals.

Section~\ref{sec:poiss-proc-funct} considers a fairly generic setting
of Poisson processes on function spaces. In the setting of
\cite{ReissWahl19} we derive the rate in the normal approximation
under very general conditions. We also show how to embed the convex
hull estimation of \cite{BaldinReiss16} in this functional
setting. Note that the normal approximation in this case was obtained
by \cite{gry19}. A further example concerns Poisson processes on the
family of hyperplanes and the related Poisson polytopes that yield
unbiased estimators of the mean width of a convex body.

\section{The hull operator}
\label{sec:hull-operator}

\subsection{Basic properties}
\label{sec:basic-properties}

Consider a Borel space $(\BX,\cX)$, see \cite{kal17} and
\cite{LastPenrose17}. We fix a {\em localising ring}
$\cX_0\subset\cX$, see \cite{kal17}.  This is a ring with the
following two properties.  First, if $B\in\cX_0$ and $C\in\cX$, then
$B\cap C\in\cX_0$.  Second, there exists a sequence $B_n\in\cX_0$,
$n\in\N$, increasing to $\BX$ such that each set from $\cX_0$ is of the
form $C\cap B_n$ for some $C\in\cX$ and some $n\in\N$.  Given a
$\sigma$-finite measure $\lambda$ on $(\BX,\cX)$ it is, for instance,
possible to choose the $B_n$ such that $\lambda(B_n)<\infty$ for each
$n\in\N$ and then to take $\cX_0$ as the sets $B\cap B_n$ for
$B\in\cX$ and $n\in\N$. A measure $\nu$ is said to be {\em locally
  finite}, if it is finite on $\cX_0$.

Let $\bN(\BX)\equiv \bN$ denote the space of all measures $\mu$ on
$\BX$ which are integer-valued on $\cX_0$.  We equip $\bN$ with the
smallest $\sigma$-field $\cN$ making the mappings $\mu\mapsto\mu(B)$
for each $B\in\cX$ measurable.  We write
$\supp\mu:=\{x\in\BX:\mu(\{x\})>0\}$ for the support of
$\mu\in\bN$. For $x\in\BX$, we write $x\in\mu$ instead of
$x\in\supp\mu$. By $\delta_x$ we denote the Dirac measure at
$x\in\BX$.  The restriction of a measure $\nu$ on $\BX$ to a set
$B\in\cX$ is denoted by $\nu_B:=\nu(\cdot\cap B)$.  For two measures
$\nu$ and $\nu'$ on $\BX$, we write $\nu'\le\nu$ if $\nu'(B)\le\nu(B)$
for each $B\in\cX$.

Consider a measurable mapping $\mu\mapsto \partial\mu$ from $\bN$ to
$\bN$ that satisfies the following properties:
\begin{enumerate}
\item[(H1)] (thinning) $\partial\mu\leq \mu$;
\item[(H2)] (additivity)  for all $\mu\in\bN$ and $x\in\partial\mu$, we have 
  \begin{equation}
    \label{eq:1a}
    \partial(\mu+\delta_x)=\partial\mu + \delta_x;
  \end{equation}
\item[(H3)] (idempotency) for all $\mu,\mu'\in\bN$ such that $\mu'\leq \mu-\partial\mu$,
  we have
  \begin{equation}
    \label{eq:1b}
    \partial(\partial\mu+\mu')=\partial\mu ; 
  \end{equation}
\item[(H4)] (consistency) if $\mu,\mu'\in\bN$ satisfy $\mu'\leq\mu$ and
  $\partial\mu=\partial\mu'$, then
  $\partial(\mu+\psi)=\partial(\mu'+\psi)$ for all $\psi\in\bN$.
\end{enumerate}

A measurable mapping $\partial:\bN\to\bN$ satisfying (H1)--(H4)
  is called a \emph{generator}.
The thinning and idempotency properties imply that $\partial\mu$ is
the minimum of $\{\mu'\in\bN:\mu'\leq\mu,
\partial\mu'=\partial\mu\}$. Indeed, if $\mu'\leq\mu$ and
$\partial\mu'=\partial\mu$, then $\partial\mu=\partial\mu'\leq \mu'$.
It will be shown in the proof of Lemma~\ref{lemma:generator-hull} that
$\partial\mu$ retains the multiplicities of its points, that is, if
$x\in\partial\mu$ is a multiple point for $\mu$, then $\partial\mu$
has the same multiplicity at $x$.

The following examples illustrate the rather abstract definition of a
generator.

\begin{example} \label{ex:ex-basic}
  Suppose that $\BX=\R^d$ and that $\cX_0$ is the system of all
  bounded Borel sets. For $\mu\in\bN$ let $\partial\mu$ be the
  restriction of $\mu$ to the vertices of the convex hull
  $\conv(\supp\mu)$. It is easy to check that this mapping is a
  generator. This example will be further discussed and generalised in
  Example~\ref{ex:chull}.
\end{example}

\begin{example}
  \label{ex:ex-basic2}
  Let $\BX$ and $\cX_0$ be as in Example~\ref{ex:ex-basic}.  For
  $\mu\in\bN$, let $\partial\mu$ be the measure supported by the
  points of $\mu$ nearest to the origin retaining their
  multiplicity. 
\end{example}

\begin{lemma}
  \label{lemma:mid-way}
  Condition (H3) is equivalent to the combination of two conditions:
  \begin{enumerate}
  \item [(H3a)] $\partial(\partial\mu)=\partial\mu$ for all $\mu\in\bN$;
  \item [(H3b)] if $\mu,\mu'\in\bN$ satisfy $\mu'\leq\mu$ and
    $\partial\mu'=\partial\mu$, then
    \begin{math}
      \partial(\mu'+\mu'')=\partial\mu
    \end{math}
    for all $\mu''\in\bN$ such that $\mu''\leq\mu-\mu'$.   
  \end{enumerate}
\end{lemma}
\begin{proof}
  Assume that (H3a) and (H3b) hold and let $\mu\in\bN$. Suppose that
  $\mu'\leq \mu-\partial\mu$. Since
  $\partial(\partial\mu)=\partial\mu$, we can apply (H3b) with
  $\partial\mu$ instead of $\mu'$ and $\mu'$ instead of $\mu''$ to
  obtain \eqref{eq:1b}

  Conversely, (H3) with $\mu'=0$ yields (H3a). Suppose that
  $\mu'\leq\mu$, $\partial\mu'=\partial\mu$ and $\mu''\leq \mu-\mu'$. 
  Since
  \begin{displaymath}
    \partial(\mu'+\mu'')=\partial(\partial\mu+\mu'-\partial\mu+\mu'')
  \end{displaymath}
  and $\mu'-\partial\mu+\mu''\leq \mu-\partial\mu$, (H3b) follows from
  (H3). 
\end{proof}

For a given generator $\partial$, 
define a measurable function $H\colon\BX\times\bN\to\R$ by
\begin{equation}
  \label{eq:H}
  H_x(\mu):= \I\big\{\partial(\mu+\delta_x)\neq \partial\mu\big\}.
\end{equation}
Further write $\bH_x(\mu):=1-H_x(\mu)$. By (H4), $\bH_x(\mu)=1$
implies that $\bH_x(\mu+\delta_y)=1$ for all $y\in\BX$. 

\begin{lemma}
  \label{lemma:minus-point}
  For all $\mu\in\bN$ and all $x\in\mu$, we have 
  $H_x(\mu-\delta_x)=H_x(\mu)$. 
\end{lemma}
\begin{proof}
  Assume that $H_x(\mu-\delta_x)=0$, that is,
  $\partial(\mu-\delta_x)=\partial\mu$. By (H4), $H_x(\mu)=0$.

  Now let $H_x(\mu)=0$, that is,
  $\partial(\mu+\delta_x)= \partial\mu$.  Hence, $x\notin\partial\mu$,
  since otherwise \eqref{eq:1a} yields a contradiction by evaluating
  the values of measures at $\{x\}$.  Furthermore, \eqref{eq:1b} with
  $\mu'=(\mu-\partial\mu-\delta_x)$ implies that
  $\partial(\mu-\delta_x)=\partial\mu$, that is,
  $H_x(\mu-\delta_x)=0$.
\end{proof}

For each $\mu\in\bN$, define the \emph{hull operator}
$\mu\mapsto[\mu]$ as 
\begin{equation}
  \label{eq:nu}
  [\mu]:=\big\{x\in\BX:\partial(\mu+\delta_x)=\partial\mu\big\}
  =\big\{x\in\BX:H_x(\mu)=0\big\}.
\end{equation}
By (H4) we have $[\mu]\subset[\mu']$ if $\mu\le\mu'$.

\begin{lemma}
  \label{lemma:generator-hull}
  Let $\partial$ be a generator. Then, for all $\mu\in\bN$,
  \begin{align}    \label{elZr}
    [\mu]&=[\partial\mu],\\
    \label{eq:partial}
    \supp\partial\mu&=\big\{x\in\mu:\partial(\mu-\delta_x)\ne \partial\mu\big\},\\
    \label{eq:5}
    \partial\mu&=\mu_{[\mu]^c}.
  \end{align}
\end{lemma}
\begin{proof}
  Assume that $x\in\BX$ satisfies
  $\partial(\partial\mu+\delta_x)=\partial\mu$. Since
  $\partial(\partial\mu)=\partial\mu$, (H4) yields that
  \begin{equation}
    \label{eq:5a}
    \partial(\partial\mu+\delta_x)=\partial(\mu+\delta_x), \quad x\in\BX.
  \end{equation}
  Hence, $\partial(\mu+\delta_x)=\partial\mu$, meaning that
  $x\in[\mu]$.  Assume, conversely, that $x\in[\mu]$. 
  By \eqref{eq:5a}, $x$
  belongs to the right-hand side of \eqref{elZr}.

  To prove \eqref{eq:partial}, take $x\in\partial\mu$ and assume that
  $\partial(\mu-\delta_x)= \partial\mu$. By
  Lemma~\ref{lemma:minus-point}, this is equivalent to
  $\partial\mu=\partial(\mu+\delta_x)$, which is impossible if
  $x\in\partial\mu$ by \eqref{eq:1a}.  Assume, conversely, that
  $\partial(\mu-\delta_x)\ne\partial\mu$ for some $x\in\mu$. We need
  to show that $x\in\partial\mu$.  However, if $x\notin\partial\mu$
  then (H3) would imply
  $\partial(\partial\mu+(\mu-\partial\mu-\delta_x))=\partial\mu$, 
  which is a contradiction.

  To prove \eqref{eq:5}, we first show that 
  $\partial\mu(\{x\})=\mu(\{x\})$ for all $x\in\partial\mu$, that is,
  the generator retains the multiplicities of its points from
  $\mu$. This follows from (H2) and (H3). Indeed, assume that $1\leq
  \partial\mu(\{x\})<\mu(\{x\})$. Then
  $\mu':=\delta_x\leq \mu-\partial\mu$, so that
  $\partial(\partial\mu+\delta_x)=\partial\mu$ by \eqref{eq:1b}. Since
  $x\in\partial\mu$ and $\partial\partial\mu=\partial\mu$,
  \eqref{eq:1a} yields that
  $\partial(\partial\mu+\delta_x)=\partial\partial\mu+\delta_x=\partial\mu+\delta_x$,
  which is a contradiction.
  
  If $x\in\partial\mu$, then $\partial\mu\neq \partial(\mu+\delta_x)$
  by (H2), hence $x\notin[\mu]$.  If $x\in\mu$ and
  $x\notin\partial\mu$, then
  $\partial(\partial\mu+\delta_x)=\partial\mu$ by (H3). By
  \eqref{elZr}, $x\in[\mu]$, so that \eqref{eq:5} holds. 
\end{proof}

The measurability of the generator
implies that $[\mu]$ is a measurable subset of $\BX$. Recall that the
map $\mu\mapsto[\mu]$ is called graph measurable if
$\{(\mu,x):x\in[\mu]\}$ is a measurable subset of $\bN\times\BX$
equipped with the product $\sigma$-algebra $\cN\otimes\cX$,
equivalently, the indicator function $(\mu,x)\mapsto \I\{x\in[\mu]\}$
is jointly measurable.

\begin{lemma}
  \label{lemma:graph}
  Let $\partial:\bN\to\bN$ be a map satisfying (H1)--(H4). Then the
  measurability of $\partial$ is equivalent to the graph measurability
  of the map  $\mu\mapsto[\mu]$.  
\end{lemma}
\begin{proof}
  Assume that $\partial$ is measurable.  By definition,
  \begin{math}
    \big\{(\mu,x):x\in[\mu]\big\}
    =\big\{(\mu,x):\partial\mu=\partial(\mu+\delta_x)\big\}.
  \end{math}
  Note that the map $(\mu,x)\mapsto\partial(\mu+\delta_x)$ from
  $\bN\times\BX\to\bN$ is measurable with respect to the product
  $\sigma$-algebra on $\bN\times\BX$, being a composition of two
  measurable maps. Furthermore, the diagonal set
  $\{(\mu,\mu')\in\bN^2:\mu=\mu'\}$ is measurable, since $\bN$ is a
  Borel space, see \cite[Theorem~1.5]{kal17}. This implies the result.

  Now assume that $\mu\mapsto[\mu]$ is graph measurable.
  For each $A\in\cX$, \eqref{eq:5} yields that
  \begin{displaymath}
    (\partial\mu)(A)=\mu\big([\mu]^c\cap A\big)
    =\int\I\{x\in A\}\I\{x\notin[\mu]\}\,\mu(dx).
  \end{displaymath}
  Since the integrand is jointly measurable, Lemma~1.15(i) in
  \cite{kal17} yields that $(\partial\mu)(A)$ is a measurable function
  of $\mu$.
\end{proof}

\begin{lemma}
  \label{lemma:stoppingset}
  Let $\partial$ be a generator, and let $\mu,\psi\in\bN$. Then
  $\mu_{[\mu]^c}=\psi$ if and only if $\mu_{[\psi]^c}=\psi$. In this
  case $[\mu]=[\psi]$.
\end{lemma}
\begin{proof}
  Recall from Lemma~\ref{lemma:generator-hull} that
  $\mu_{[\mu]^c}=\partial\mu$ and assume first that
  $\partial\mu=\psi$.  By \eqref{elZr},
  \begin{align*}
    \mu_{[\psi]^c}=\mu_{[\partial\mu]^c}=\mu_{[\mu]^c}=\partial\mu=\psi.
  \end{align*}
  To prove the converse, assume that $\mu_{[\psi]^c}=\psi$. Then
  $\psi\leq\mu$ and it follows from (H4) that
  $[\psi]\subset[\mu]$. Therefore, $\partial\mu=\mu_{[\mu]^c}\le \mu_{[\psi]^c}=\psi$.
  Since $\psi\le \mu$, we have $\psi=\partial\mu+\psi'$ for some
  $\psi'\le \mu-\partial\mu$.  By (H3), $\partial\psi=\partial\mu$.
  Since $\psi$ is supported by $[\psi]^c$, we have 
  $\partial\psi=\psi_{[\psi]^c}=\mu_{[\psi]^c}=\psi$. Thus,
  $\psi=\partial\mu$. 
  By \eqref{elZr}, $[\mu]=[\partial\mu]=[\psi]$.
\end{proof}

\begin{lemma}
  \label{lemma:T}
  For all $x,y\in\BX$ and $\mu\in\bN$, the following statements are
  equivalent:
  \begin{enumerate}
  \item [(i)]
    $\partial(\mu+\delta_x+\delta_y)=\partial(\mu+\delta_x)=\partial(\mu+\delta_y)$,
  \item [(ii)] $\partial(\mu+\delta_x+\delta_y)=\partial(\mu+\delta_x)=\partial\mu$,
  \item [(iii)] $\partial(\mu+\delta_x)=\partial(\mu+\delta_y)=\partial\mu$.
  \end{enumerate}
\end{lemma}
\begin{proof}
  (i)$\Rightarrow$(ii) For $x,y\in\BX$, define
  \begin{equation}
    \label{eq:eta-prime}
    \mu':=\partial(\mu+\delta_x+\delta_y)=\partial(\mu+\delta_x). 
  \end{equation}
  If $y\in\mu'$, then \eqref{eq:1a} yields that
  $\mu'=\partial(\mu+\delta_x)+\delta_y$, a contradiction to
  \eqref{eq:eta-prime}. The same applies to $x$. Hence, we can assume
  $x,y\notin \mu'$. 
  Then $\mu-\mu'\geq 0$. Since
  $\partial\mu'=\partial(\mu+\delta_x+\delta_y)$,
  Lemma~\ref{lemma:mid-way} yields that
  \begin{displaymath}
    \partial\mu= \partial(\mu'+(\mu-\mu'))=\partial(\mu+\delta_x+\delta_y).
  \end{displaymath}

  \noindent
  (ii)$\Rightarrow$(iii) Since
  $\partial\mu=\partial(\mu+\delta_x+\delta_y)$,
  Lemma~\ref{lemma:mid-way} yields that $\partial\mu=\partial(\mu+\delta_y)$.

  \noindent
  (iii)$\Rightarrow$(i) follows from (H4). 
\end{proof}

The following result extends a part of Lemma~\ref{lemma:T} to several
points.

\begin{lemma}
  \label{lemma:cyclic}
  Let $x_1,\dots,x_m\in\BX$ for some $m\geq 2$. Then
  $\mu\in\bN$ satisfies
  \begin{displaymath}
    \partial(\mu+\delta_{x_1})=\partial(\mu+\delta_{x_1}+\delta_{x_2}),
    \dots,\partial(\mu+\delta_{x_m})=\partial(\mu+\delta_{x_m}+\delta_{x_1})
  \end{displaymath}
  if and only if 
  \begin{math}
    \partial(\mu+\delta_{x_1})=\cdots=\partial(\mu+\delta_{x_m})=\partial\mu.
  \end{math}
\end{lemma}
\begin{proof}
  Sufficiency immediately follows from (H4).  For the proof of
  necessity, denote $\delta_\bx:=\delta_{x_1}+\cdots+\delta_{x_m}$. By
  (H4)
  \begin{displaymath}
    \mu':=\partial(\mu+\delta_\bx)=\partial(\mu+\delta_\bx-\delta_{x_i}),\quad
    i=1,\dots,m.
  \end{displaymath}
  Then $x_i\notin\mu'$ for all $i=1,\dots,m$, so that
  $\mu-\mu'\geq0$. Since $\partial\mu'=\partial(\mu+\delta_\bx)$,
  Lemma~\ref{lemma:mid-way} yields that
  $\partial\mu=\partial\big(\mu'+(\mu-\mu')\big)=\partial(\mu+\delta_\bx)$.
  Hence, $\partial\mu=\partial(\mu+\delta_{x_i})$ for all
  $i=1,\dots,m$.
\end{proof}

A trivial example of a generator is $\partial\mu:=\mu$.  In this case,
$[\mu]=\emptyset$ for all $\mu\in\bN$.  The following is the most
standard nontrivial example of a generator.

\begin{example}
  \label{ex:chull}
  Let $\BX$ be an open subset of $\R^d$. For $\mu\in\bN$, define
  $\partial\mu$ to be the restriction of $\mu$ to the extreme points
  (vertices) of the convex hull of the support of $\mu$. The
  properties (H1)--(H4) are easy to check. By Lemma~\ref{lemma:graph},
  the measurability of $\partial$ follows from the graph measurability
  of the corresponding hull operator given by the convex hull of the
  support of $\mu$ with eliminated vertices.  It suffices to assume
  that $\mu$ is finite, since our hull operator satisfies
  $[\mu_{B_n}]\uparrow [\mu]$ if $B_n\in\cX_0$ and $B_n\uparrow\BX$.
  Suppose that $\mu=\delta_{x_1}+\cdots+\delta_{x_n}$ for some
  $x_1,\ldots,x_n\in\R^d$, and let $x\in\R^d$. Then $x\in[\mu]$ if and
  only if there exist $k\le d+1$ and $i_1,\ldots,i_k\le n$ such that
  $x\in \mathrm{rel\,int}\conv(\{x_{i_1},\dots,x_{i_k}\})$, where
  $\mathrm{rel\,int}\conv(\{x_{i_1},\dots,x_{i_k}\})$ is the relative
  interior of the convex hull of $x_{i_1},\dots,x_{i_k}$.  The mapping
  $(x,y_1,\ldots,y_m)\mapsto \I\big\{x\in
  \mathrm{rel\,int}\conv(\{y_1,\dots,y_m\})\big\}$ is measurable for each
  $m\in\N$. The asserted measurability of
  $(x,\mu)\mapsto \I\{x\in[\mu]\}$ follows from the fact, that the
  points of $\mu$ can be numbered in a measurable way, see \cite[Corollary~6.5]{LastPenrose17}.
  \begin{enumerate}
  \item [(i)] Assume that $\BX=\R^d$ and that $\cX_0$ is the family of
    all bounded Borel sets in $\BX$. This is the setting of
    Example~\ref{ex:ex-basic}.
    If $\cX_0$ is the family of all Borel sets, then $\mu$ is
      finite and we arrive at the setting of \cite{BaldinReiss16}.
  \item [(ii)] Assume that $\BX$ is a proper open cone in $\R^d$ and $\cX_0$ is the
    family of relatively compact subsets of $\BX$. Then $\mu$ may be
    infinite and its support may have a concentration point at the
    origin. As a result, the generator $\partial\mu$ may contain
    infinitely many points. For instance, this is the case if $\mu$
    is a realisation of a homogeneous Poisson process on
    $\BX=(0,\infty)^d$.
  \item [(iii)] Assume that $\BX$ is a proper open cone in $\R^d$ and $\mu$ is a
    realisation of the Poisson process with intensity
    $\|x\|^{-d}$. Such $\mu$ has a concentration point at the origin
    and infinitely many points in the complement to any ball. In this
    case, the convex hull of $\mu$ is the whole $\BX$, hence,
    $\partial\mu=0$. 
  \end{enumerate}
\end{example}

\begin{example}
  Let $K$ be a convex compact set, and let
  $\BX:=\{(s,u):s\in K, 0\leq u\leq \rho(s,K^c)\}$ be a subset of
  $K\times[0,\infty)$, where $\rho(s,K^c)$ is the distance from $s$ to
  the complement of $K$. Denote by $B_u(s)$ the Euclidean ball of
  radius $u$ centred at $s$. For $\mu\in\bN$, define $\partial\mu$ as
  the restriction of $\mu$ onto $(s,u)\in\mu$ such that the union of
  $B_u(s)$ for $(s,u)\in\mu$ is not equal to the union of $B_u(s)$ for
  $(s,u)\in \mu-\mu(\{(s,u)\})\delta_{(s,u)}$.
\end{example}

\subsection{Difference operators}
\label{sec:difference-operators}

In order to apply stochastic calculus tools, it is necessary to find
out how a functional $G:\bN\to\R$ changes under addition of extra
points to its argument.  The first order difference is defined by
\begin{displaymath}
  D_x G(\mu):=G(\mu+\delta_x)-G(\mu),
\end{displaymath}
and higher order ones are defined by iterating
\begin{displaymath}
  D^{m+1}_{x_1,\dots,x_m,x_{m+1}} G(\mu):=D^m_{x_1,\dots,x_m}
  G(\mu+\delta_{x_{m+1}})-D^m_{x_1,\dots,x_m} G(\mu).
\end{displaymath}

These constructions will be often applied to $H_z(\mu)$ considered 
as a function of $\mu$.  Then
\begin{displaymath}
  D_x H_z(\mu)
  =\I\big\{\partial(\mu+\delta_x)\neq \partial(\mu+\delta_x+\delta_z)\big\}
  -\I\big\{\partial\mu\neq \partial(\mu+\delta_z)\big\}.
\end{displaymath}
If $\partial\mu=\partial(\mu+\delta_z)$, then also $\partial(\mu+\delta_x)=
\partial(\mu+\delta_x+\delta_z)$ by (H4). Hence, 
\begin{equation}
  \label{e3.8}
  D_x H_z(\mu)=-H_z(\mu)\bH_z(\mu+\delta_x). 
\end{equation}
The higher order differences can be found by induction as
\begin{align}  \label{eq:16}
  D^m_{x_1,\dots,x_m} H_z(\mu)
  =(-1)^mH_z(\mu)\bigg[\sum_{k=1}^m (-1)^{k-1}\sum_{1\leq j_1<\cdots<j_k\leq
  m} \bH_z(\mu+\delta_{x_{j_1}}+\cdots+\delta_{x_{j_k}})\bigg].
\end{align}

\begin{lemma}
  \label{lemma:cyclic-bis}
  For all $m\geq 2$ and $z_1,\dots,z_m\in\BX$,
  \begin{equation}
    \label{eq:cyclic}
    D_{z_1}H_{z_2}(\mu)D_{z_2}H_{z_3}(\mu)\cdots D_{z_m}H_{z_1}(\mu)=0,\qquad \mu\in\bN. 
  \end{equation}
\end{lemma}
\begin{proof}
  If the product in \eqref{eq:cyclic} does not vanish, then, by
  \eqref{e3.8},
  \begin{displaymath}
    \bH_{z_2}(\mu+\delta_{z_1})\cdots \bH_{z_1}(\mu+\delta_{z_m})=1.
  \end{displaymath}
  By Lemma~\ref{lemma:cyclic}, $H_{z_1}(\mu)=\cdots=H_{z_m}(\mu)=0$. 
\end{proof}

Note that Lemma~\ref{lemma:T} implies that 
\begin{equation}
  \label{eq:17}
  \bH_y(\mu+\delta_x)\bH_x(\mu+\delta_y)=\bH_x(\mu) \bH_y(\mu).
\end{equation}
Together with (H4), \eqref{eq:17} implies \eqref{eq:cyclic} for $n=2$.
In all interesting cases, the generator is nontrivial in the sense that
\begin{equation}
  \label{eq:nontrivial}
  \partial\delta_x=\delta_x, \quad x\in\BX,
\end{equation}
equivalently, $\bH_x(0)=1$.
If this property holds, letting
$\mu=0$ in \eqref{eq:17} yields $\bH_x(\delta_y)\bH_y(\delta_x)=0$.

\subsection{Generators with the prime property}
\label{sec:prime-property}

The generator is said to satisfy the \emph{prime property} if the
corresponding function $H$ satisfies
\begin{equation}
  \label{eq:6}
  H_z(\mu)=\prod_{x\in\mu} H_z(\delta_x) 
\end{equation}
for all $\mu\in\bN$, equivalently,
$\bH_z(\mu)=\max_{x\in\mu} \bH_z(\delta_x)$. Note that (H4) always
implies that $H_z(\mu)$ is dominated by the product on the right-hand
side of \eqref{eq:6}, equivalently,
$\bH_z(\mu)\geq \max_{x\in\mu} \bH_z(\delta_x)$.  If a generator
satisfies the prime property, then \eqref{e3.8} simplifies to
\begin{math}
  D_x H_z(\mu)=-H_z(\mu)\bH_z(\delta_x).
\end{math}
The prime property substantially simplifies many formulas, for
instance, the forthcoming variance formula \eqref{evarianceprime} (see
also Lemma~\ref{lemma:H-bound}) or the bounds on the normal
approximation in Corollary~\ref{cor:normal-prime}.  It does not hold
in the convex hull example, but holds in the function setting of
Lemma~\ref{lemma:domination}.

\begin{lemma}
  \label{lemma:prime-c}
  A generator is prime if and only if the corresponding hull operator
  satisfies
  \begin{equation}
    \label{eq:prime-brackets}
    [\mu]=\bigcup_{x\in\mu}[\delta_x].
  \end{equation}
\end{lemma}
\begin{proof}
  If \eqref{eq:6} holds, then
  \begin{displaymath}
    [\mu]=\bigg\{z: \prod_{x\in\mu} H_z(\delta_x)=0\bigg\}
    =\bigcup_{x\in\mu} \big\{z:H_z(\delta_x)=0\big\}.
  \end{displaymath}
  In the other direction, if \eqref{eq:prime-brackets} holds, then
  $H_z(\mu)$ equals the product of $\I\{z\notin[\delta_x]\}$ over
  $x\in\mu$.
\end{proof}

\begin{remark} \label{rem:prime}
  Assume that the generator satisfies
  \eqref{eq:nontrivial}.
  By Lemma~\ref{lemma:T}, letting
  $y\prec x$ whenever $\bH_y(\delta_x)=1$ for $x,y\in\BX$ defines a
  strict partial order on $\BX$. Indeed, the antisymmetry follows from
  $\bH_x(\delta_y)\bH_y(\delta_x)=0$. Furthermore, if $y\prec x$ and
  $z\prec y$, then $\bH_y(\delta_x)=\bH_z(\delta_y)=1$, i.e.,
  $\partial(\delta_y+\delta_x)=\partial(\delta_x)$ and
  $\partial(\delta_y+\delta_z)=\partial(\delta_y)$. By (H4),
  $\partial(\delta_y+\delta_x+\delta_z)=\partial(\delta_x+\delta_z)$
  and
  $\partial(\delta_y+\delta_z+\delta_x)=\partial(\delta_y+\delta_x)$.
  Hence $\partial(\delta_x+\delta_z)=\partial(\delta_x)$, which means
  that $z\prec x$. Therefore $\prec$ is transitive.  The prime
  property of $\partial$ corresponds to the prime property of the
  order relation, see \cite[Proposition~I-3.12]{gie:03}.
\end{remark}

\begin{example}
  Let $\BX=[0,1]^d$. Let $\partial\mu$ be the set of Pareto optimal
  points for $\mu$, that is, points $x\in\mu$ (with retained
  multiplicities) which do not dominate coordinatewisely any
  other point from the support of $\mu$.
  This generator has the prime property. In this case, $[\mu]$ is the
  set of points $y\in\BX$ such that $y\notin\partial\mu$ and $x\leq y$
  coordinatewisely for at least one $x\in\mu$.
\end{example}

\section{Spatial strong Markov property}
\label{secstopping}

A {\em point process} on $\BX$ is a random element of $\bN$, defined
over a given probability space $(\Omega,\cF,\BP)$.  In this paper we
assume that $\eta$ is a {\em Poisson process} $\eta$ with intensity
measure $\lambda$, see \cite{LastPenrose17}, where $\lambda$ is
assumed to be locally finite, that is finite on $\cX_0$.  Its distribution is denoted by
$\Pi_\lambda$.

Consider a generator $\partial\colon\bN\to\bN$ and the
corresponding hull operator.  First,
we need to confirm that the restriction of a point process $\eta$ to the hull $[\eta]$
is indeed a point process.

\begin{lemma}
  \label{lemma:eta-brackets-measurable}
  If $\eta$ is a point process on $\BX$, then its restriction
  $\eta_{[\eta]}$ to $[\eta]$ is also a point process.
\end{lemma}
\begin{proof}
  By measurability of the generator, $\partial\eta$ is a point
  process. Then it suffices to note that
  $\eta_{[\eta]}=\eta-\eta_{[\eta]^c}$ and
  $\eta_{[\eta]^c}=\partial\eta$ by \eqref{eq:5}.
\end{proof}

From now on we assume that $\eta$ is a Poisson process
with intensity measure $\lambda$. Assume that
$\partial$ is a generator that satisfies properties (H1)--(H4).
The following result provides an integral representation for the
distribution of $(\partial\eta,\eta_{[\eta]})$ and also shows that
$[\eta]$ satisfies a \emph{strong Markov property} for stopping sets, see
\cite{Zuyev99}. 

Given $x_1,\dots,x_n\in\BX$, we write $\bx:=(x_1,\dots,x_n)$ and
$\delta_{\bx}:=\delta_{x_1}+\cdots+\delta_{x_n}$. Define the set
$[\bx]:=[\delta_{\bx}]$ by \eqref{eq:nu}.  Further, let $C_n$ be the
set of all $\bx\in\BX^n$ such that
$\partial\delta_{\bx}=\delta_{\bx}$.

\begin{theorem}
  \label{tstopping1} 
  For each $n\in\N$, 
  \begin{equation}\label{e3.6}
    \BE\Big[\I\big\{\partial\eta(\BX)=n\big\}
    \I\big\{(\partial\eta,\eta_{[\eta]})\in\cdot\big\}\Big]
    =\frac{1}{n!}\BE\int_{C_n} 
    \I\big\{(\delta_{\bx},\eta_{[\bx]})\in\cdot\big\}
    \exp\big[-\lambda([\bx]^c)\big]\, \lambda^n(d\bx).
  \end{equation}
  Assume that $\partial\eta$ is almost surely finite.
  Then
  \begin{align}
    \label{stopping3}
    \BP\big(\eta_{[\eta]} \in\cdot\mid \partial\eta\big)
    =\Pi_{\lambda_{[\eta]}}(\cdot),\quad \BP\text{-a.s.}
  \end{align}
\end{theorem}
\begin{proof}
  We apply Theorem A.3 and equation (A.5) from \cite{osss21} to the mapping
  $\mu\mapsto Z(\mu):=[\mu]^c$.  Let $\mu,\psi\in
  \bN$. Lemma~\ref{lemma:stoppingset} says that $\mu_{Z(\mu)}=\psi$ if
  and only if $\mu_{Z(\psi)}=\psi$, in which case $Z(\mu)=Z(\psi)$.
  By \cite[Remark~A.4]{osss21}, we obtain \eqref{e3.6} from
  \cite[(A.5)]{osss21} and \eqref{stopping3} from
  \cite[(A.4)]{osss21}.
\end{proof}

In order to extend the strong Markov property \eqref{stopping3} to possibly
infinite generators, we need to impose continuity conditions.  The
following result provides an example of such conditions.  
Recall the definition of the sequence $(B_n)$ and the ring $\cX_0$ 
from Section~\ref{sec:basic-properties}. The following
result is proved in the Supplement \cite{las:mol24s}. 

\begin{proposition}  \label{prop:infinite}
  Assume for each $B\in\{B_m:m\in\N\}$ that
  \begin{align}\label{eH6}
    &    \lim_{n\to\infty}\BP\big((\partial \eta_{B_n})_B=(\partial \eta)_B,
      \eta_{[\eta_{B_n}]\cap B}=\eta_{[\eta]\cap B}\big)=1,\\ \label{eH61}
    &\lim_{n\to\infty}\BE\int\I\big\{\mu_{[\eta]\cap B}
      = \mu_{[\eta_{B_n}]\cap B}\big\}\Pi_\lambda(d\mu)=1. 
  \end{align}
  Then \eqref{stopping3} holds. 
\end{proposition}

\begin{example}[Convex hull of a finite Poisson process]
  \label{ex:chull-stopping}
  Consider Example~\ref{ex:chull}(i) 
  and assume that $\lambda(\R^d)<\infty$.  Theorem~\ref{tstopping1}
  yields that \eqref{stopping3} holds and, for each $n\in\N$,
  \begin{align*}
    \BP\big(\partial\eta(\R^d)=n,\partial\eta\in\cdot\big)
    =\frac{1}{n!}e^{-\lambda(\R^d)}
    \int_{C_n}\I\{\delta_{\bx}\in\cdot\}
    \exp\big[\lambda(\conv(\{x_1,\dots,x_n\}))\big]\,
    \lambda^{n}(d\bx).
  \end{align*}
  The set $C_n$ consists of all $n$-tuples of points that come up as
  vertices of their convex hull.  
  In the special case of a diffuse intensity measure on the unit ball,
  the result \eqref{stopping3} was mentioned in \cite{Privault12}.
  For a homogeneous Poisson point process on a convex body a proof is
  given in \cite{BaldinReiss16} based on stopping set and spatial
  martingale arguments.
\end{example}

\begin{example}[Random polytopes]
  \label{ex:polytope-stopping}
  Let $\BX$ be the affine Grassmannian $A(d,d-1)$, that is, the family
  of all $(d-1)$-dimensional planes in $\R^d$.  Let $H\in
  A(d,d-1)$. If $0\notin H$ we denote by $H^-$ the (closed)
  half-space in $\R^d$ bounded by $H$ such that $0\in H^-$. If $0\in H$ we let
  $H^-:=H$.  For a counting measure $\mu$ on $A(d,d-1)$, define $P_\mu$
  as the intersection of $H^-$ for all $H\in\mu$, and let
  $P_0:=\R^d$. Further, define a generator $\partial\mu$ as the set of
  $H\in\mu$ (with the multiplicites retained) such that $H\cap P_\mu$
  has dimension $d-1$. The corresponding hull operator $[\mu]$ is the
  set of all $H\in A(d,d-1)$ such that $H\cap P_\mu=\emptyset$ or $H$
  contains faces of $P_\mu$ of dimension at most $d-2$.  Let $\eta$ be
  a Poisson process on $A(d,d-1)$ with a diffuse intensity measure
  $\lambda$ such that 
  \begin{math}
    \lambda(\{H: H\cap K\ne\emptyset\})<\infty
  \end{math}
  for each convex body $K$.  We also assume that $\lambda$ does not
  charge the family of hyperplanes which pass through the origin.  The
  random convex set $P_\eta$ is called the {\em Poisson polytope}.
  Since $\lambda$ is diffuse, the hull $[\eta]$ is a.s.\ the family of
  all hyperplanes that do not hit $P_\eta$. The generator
  $\partial\eta$ consists of planes intersecting the boundary of
  $P_\eta$ at $(d-1)$-dimensional facets, and the cardinality of
  $\partial\eta$ is a.s.\ finite, see \cite{MR4134241}. By
  Theorem~\ref{tstopping1}, $\eta$ restricted to $[\eta]$ and
  conditional upon $P_\eta$ a.s.\ coincides with the distribution of
  the Poisson process having the intensity $\lambda$ restricted to
  $\{H\in A(d,d-1):H\cap P_\eta=\emptyset\}$.
\end{example}

\section{Poisson hull estimator}
\label{secestimator}

As before, let $\eta$ be a Poisson process with intensity measure
$\lambda$. For a given function $f\in L^1(\lambda)$,
consider the integral
\begin{displaymath}
  F:=\int f(x)\lambda(dx).
\end{displaymath}
The random variable
\begin{displaymath}
  F^*:=\int f(x)\,\eta(dx)
\end{displaymath}
is known as a {\em linear functional} of $\eta$.  By Campbell's
formula, $\BE F^*=F$. Therefore, if $\eta$ can be observed, then $F^*$
is an unbiased estimator of $F$, providing some information on
$\lambda$.  By \cite[Lemma~12.2]{LastPenrose17},
\begin{align}
  \label{evarsimple}
  \BV F^*=\int f(x)^2\,\lambda(dx).
\end{align}
In the following sections we often assume $f\in L^2(\lambda)$, to
ensure that $\BV F^*<\infty$.

Assume now that $\partial$ is a generator.  If $\partial\eta$ is not
necessarily finite, we assume that the strong Markov property
\eqref{stopping3} holds, see also Proposition~\ref{prop:infinite} and
Lemma~\ref{lMarkovfunctions}.  The conditional expectation
\begin{equation}
  \label{defestimator}
  \hat{F}:=\BE[F^*\mid \partial\eta]
\end{equation}
can be used as an unbiased estimator of $F$.  We call it a
\emph{Poisson hull estimator} (based on the chosen generator
$\partial$) and note that
$\hat{F}\equiv\hat{F}(\eta)\equiv \hat{F}(\eta,f)$ depends on $\eta$
and $f$.  As we will see, this estimator requires $\lambda$ to be
known on $[\eta]$. This knowledge designates our estimator as the
oracle one. 
We will also see, that it generalises both the oracle estimator for
the volume of a convex body from \cite{BaldinReiss16} as well as the
estimator for integrals of H\"older functions studied in
\cite{ReissWahl19}.  There are many more interesting special cases.

If $f\in L^2(\lambda)$ we can use the conditional variance formula to
see that $\hat{F}$ has a smaller variance than $F^*$.  The
most important setting arises when $\lambda$ is the restriction of a
known measure $\bar\lambda$ to an unknown set $A$ from a certain
system $\tilde\cX\subset\cX$.  The next remark shows that, under a
natural assumption, $\partial\eta$ becomes a \emph{sufficient
  statistic} for the \emph{parameter} $A$.

\begin{remark}
  \label{rem:T-sufficient}
  Let $\bar\lambda$ be a
  measure on $\BX$, and let $\tilde\cX\subset\cX$ be a subfamily such
  that the measure $\bar\lambda$ restricted to any $A\in\tilde\cX$ is
  locally finite and that the strong Markov property holds for a
  Poisson process with this intensity measure.  For instance, this is
  the case if $\bar\lambda(A)<\infty$ for each $A\in\tilde\cX$. Assume
  that $\lambda=\bar\lambda_A$ is the restriction of $\bar\lambda$ to
  some $A\in\tilde\cX$.  We consider the set $A\in\tilde\cX$ as a
  parameter and denote the expectation with respect to a Poisson
  process $\eta$ of intensity $\bar\lambda_A$ by $\BE_A$.  Consider a
  generator $\partial$ and assume that the associated hull operator
  satisfies
  \begin{align}\label{eq:suffstat}
    \bar\lambda\big([\eta]\setminus A\big)=0,\quad \BP_A\text{-a.s.},\,A\in\tilde{X}.
  \end{align}
  Then $\partial\eta$ is a sufficient
  statistic for the parameter $A$.  Indeed, by the strong Markov
  property \eqref{stopping3} the conditional distribution
  $\BP_A(\eta\in\cdot\mid\partial\eta)$ depends $\BP_A$-a.s.\ only on
  $(\bar\lambda_A)_{[\eta]}$.  By assumption \eqref{eq:suffstat} we
  have $\BP_A$-a.s.\ that
  $(\bar\lambda_A)_{[\eta]}=\bar\lambda_{[\eta]}$, which does not
  depend on $A$.
\end{remark}

\begin{example}
  \label{e:genoracle}
  This example illustrates Remark~\ref{rem:T-sufficient}.  Assume that
  $\BX=\R^d$ and that $\cX_0$ is the system of all bounded Borel
  sets. Fix $B\in\cX_0$, and let $\tilde{\cX}$ be the system of all
  sets of the form $A=K\setminus B$, where $K\subset\R^d$ is convex
  and compact.  Let $\bar\lambda$ be a locally finite measure on
  $\R^d$ with $\bar\lambda(B)=0$.  Let $\partial$ be the convex hull
  generator from Example~\ref{ex:ex-basic}.  For $A=K\setminus B$ as
  above we have $\BP_A([\eta]\subset K)=1$.  Therefore, we have
  $\bar\lambda([\eta]\setminus A)=\bar\lambda([\eta]\cap B)=0$,
  $\BP_A$-a.s.,
  so that \eqref{eq:suffstat} holds. If $B=\emptyset$ and
  $\bar{\lambda}$ is a known multiple of the Lebesgue measure, then we
  recover the setting of \cite{BaldinReiss16}.
\end{example} 

\begin{lemma}
  \label{lemma:hat-F-1}
  We have 
  \begin{align}
    \label{e6.23}
    \hat{F}=\int f(x)\,\lambda_{[\eta]}(dx)+\int f(x)\,\partial\eta(dx)
    \quad \text{a.s.}
  \end{align}
\end{lemma}  
\begin{proof}
  By Lemma~\ref{lemma:generator-hull},
  $\eta=\partial\eta+\eta_{[\eta]}$ a.s.  Hence,
  \begin{displaymath}
    \hat{F}
    =\BE\bigg[\int f(x)\,\eta_{[\eta]}(dx)\Bigm\vert \partial\eta\bigg]
    +\int f(x)\,\partial\eta(dx).
  \end{displaymath}
  The statement follows from the assumed strong Markov property
  \eqref{stopping3}; see
  Theorem~\ref{tstopping1} in case of a finite generator and 
  Proposition~\ref{prop:infinite} for a possibly infinite
  generator $\partial\eta$. 
\end{proof}

Note that the almost sure restriction in \eqref{e6.23} stems from
Theorem~\ref{tstopping1}.  The next result follows from
Lemma~\ref{lemma:hat-F-1} and Lemma~\ref{lemma:generator-hull}.

\begin{lemma}
  \label{lemma:hat-F}
  We have
  \begin{align}
    \label{defestimator-bis}
    \hat{F}-F=
    \int f(z) H_z(\eta-\delta_z)\,\eta(dz)
    -\int f(z) H_z(\eta)\,\lambda(dz) \quad \text{a.s.}
  \end{align}
\end{lemma}

\begin{example}[Convex hulls]
  \label{ex:chull-oracle}
  Consider Example~\ref{ex:chull}(i).  Then $[\eta]$ equals the convex
  hull of the support of $\eta$ with the vertices removed.  In the
  special case of $\lambda$ being the Lebesgue measure $V_d$
  restricted to a convex body $K\subset\R^d$ and $f\equiv 1$,
  \eqref{e6.23} provides an estimator for the Lebesgue measure of $K$.
  It reads $\hat{F}=V_d([\eta])+M$, where $M:=\partial\eta(\R^d)$ is
  the cardinality of $\partial\eta$, that is, the number of vertices
  of the convex hull of $\supp\eta$.  This is the oracle estimator for
  the volume of a convex body discovered by \cite{BaldinReiss16}.  In
  the more general case of Example \ref{e:genoracle} (still taking
  $f\equiv 1$) we obtain that $\hat{F}=\bar{\lambda}([\eta])+M$ is an
  unbiased estimator of $\bar{\lambda}(K)$. If $\bar\lambda$ is not
  diffuse, then $M$ includes possible multiplicities.
\end{example}

\begin{example}\label{ex:minimum}
  Let $\eta$ be a homogeneous unit intensity Poisson process on
  $\BX=[a,\infty)$ with $a\in\R$, so that $\lambda$ is the Lebesgue
  measure. Define 
  $\partial\mu:=\mu(\{\zeta(\mu)\})\delta_{\zeta(\mu)}$, where
  $\zeta(\mu):=\min\supp\mu$; note that the minimum is attained and,
  since $\lambda$ is diffuse, there are no multiple points at
  $\zeta:=\zeta(\eta)$. Then $[\eta]=(\zeta,\infty)$. Let $f$
  be an integrable function on $\BX$. Then
  \begin{displaymath}
    \hat{F}=\int_\zeta^\infty f(x)\,dx+f(\zeta)
  \end{displaymath}
  is an unbiased estimator of $F=\int_a^\infty f(x)\,dx$. For instance,
  if $f(x)=(p-1)x^{-p}$ with $p>1$, then $F=a^{1-p}$ is estimated by 
  $\hat{F}=\zeta^{1-p}(1+(p-1)\zeta^{-1})$. This example can be seen
  as a special case of Example~\ref{ex:ex-basic2}.
\end{example}

Suppose that $G\colon \BX\times\bN\to\R$ is measurable and satisfies
$\BE\int |G(x,\eta)|\,\lambda(dx)<\infty$.  Then
\begin{displaymath}
  \deltaKS(G):=\int G(x,\eta-\delta_x)\,\eta(dx)-\int G(x,\eta)\,\lambda(dx)
\end{displaymath}
is said to be the {\em Kabanov--Skorohod integral} of $G$ with respect
to $\eta$. In fact this is a pathwise version of a Malliavin operator,
see \cite[Theorem~5]{Last16}.  By the Mecke equation,
$\BE \deltaKS(G)=0$.  Equation \eqref{defestimator-bis} can be
rewritten as
\begin{equation}
  \label{estimatorKS}
  \hat{F}-F=\deltaKS(fH),\quad \BP\text{-a.s.}
\end{equation}

Let $x\in\BX$. By \eqref{e3.8} and \eqref{defestimator-bis},
\begin{align*} 
  D_x\hat{F}=f(x)H_x(\eta)
  -\int f(z) H_z(\eta-\delta_z)\bH_z(\eta+\delta_x-\delta_z)\eta(dz)
  +\int f(z) H_z(\eta)\bH_z(\eta+\delta_x) \lambda(dz).
\end{align*}
The two latter terms constitute the Kabanov--Skorohod integral of
$(z,\mu)\mapsto f(z)\bH_z(\mu+\delta_x)H_z(\mu)$. This can be
generalised to obtain higher order differences, leading to the chaos
expansion of $\hat{F}$ as described in \cite[Eq.~(36)]{Last16} in
terms of these differences.

\section{Variance of the Poisson hull estimator}
\label{sec:vari-poiss-hull}


As before, we consider a generator which satisfies properties
(H1)--(H4) and a Poisson process $\eta$ with intensity measure
$\lambda$.  In case the generator is infinite, impose additionally
that \eqref{stopping3} holds.  We consider the Poisson hull estimator
\eqref{defestimator}, where it is now assumed that
$f\in L^1(\lambda)\cap L^2(\lambda)$.  The variance of $\hat{F}$ is
calculated as follows.

\begin{theorem}\label{tvariance} 
  Let $f\in L^1(\lambda)\cap L^2(\lambda)$.  
  Then $\hat{F}$, defined at \eqref{defestimator}, is
  square integrable and satisfies
  \begin{align}
    \label{e3.47}
    \BV \hat{F}=\int f(x)^2\BP\big(\partial(\eta+\delta_x)\ne
    \partial\eta\big)\,
    \lambda(dx).
  \end{align}
\end{theorem}
\begin{proof}
  By the conditional variance formula
  and \eqref{e6.23},
  \begin{align*}
    \BV \hat{F}&=\BV F^*-\BE\BV[F^*\mid \partial\eta]\\
    &=\int f(x)^2\,\lambda(dx)-\BE\BV\bigg[\int
      f(x)\,\eta_{[\eta]}(dx)+\int
      f(x)\,\partial\eta(dx) \Bigm\vert \partial\eta\bigg]\\
    &=\int f(x)^2\,\lambda(dx)
    -\BE\BV\bigg[\int f(x)\,\eta_{[\eta]}(dx) \Bigm\vert \partial\eta\bigg].
  \end{align*}
  By Theorem~\ref{tstopping1} and since $\bH_x(\eta)=\I\{x\in[\eta]\}$, 
  \begin{displaymath}
    \BE\BV\bigg[\int f(x)\,\eta_{[\eta]}(dx) \Bigm\vert\partial\eta\bigg]
    =\BE\int f(x)^2\bH_x(\eta)\,\lambda(dx). \qedhere
  \end{displaymath}
\end{proof}

By \eqref{estimatorKS} and under a suitable moment assumption, the
isometry property of Kabanov--Skorohod integrals from \cite[Theorem~5]{Last16}
yields that
\begin{displaymath}
  \BE \deltaKS(fH)^2=\BE \int f(x)^2 H_x(\eta)^2\,\lambda(dx)
  +\BE \int \big(D_y (f(x)H_x(\eta))\big)
  \big(D_x (f(y)H_y(\eta))\big)\,\lambda^2(d(x,y)).  
\end{displaymath}
The second integrand equals $f(x)f(y)D_yH_x(\eta)D_xH_y(\eta)$ and so
vanishes by Lemma~\ref{lemma:cyclic-bis}. In view of
$\BE\deltaKS(fH)=0$, this provides an alternative proof of the
variance formula \eqref{e3.47}.

In the sequel, \eqref{e3.47} is usually written as 
\begin{displaymath}
  \BV \hat{F}=\int f(x)^2\BE H_x(\eta)\,\lambda(dx).
\end{displaymath}
By the polarisation identity, \eqref{e3.47} yields the following result,
which can be alternatively derived directly from the Mecke equation.

\begin{proposition}
  \label{prop:two-functions}
  If $f,g\in L^1(\lambda)\cap L^2(\lambda)$, then
  \begin{displaymath}
    \BC\big[\hat{F}(\eta,f),\hat{F}(\eta,g)\big]
    =\int f(x)g(x)\BE H_x(\eta)\,\lambda(dx).
  \end{displaymath}
\end{proposition}

The prime property yields a simpler expression for the variance of
$\hat{F}$. By \cite[Exercise~3.7]{LastPenrose17},
\begin{displaymath}
  \BE H_x(\eta)=\BE \prod_{y\in\eta} H_x(\delta_y)
  =\exp\Big[-\int \bH_x(\delta_y)\lambda(dy)\Big], 
\end{displaymath}
so that 
\begin{displaymath}\label{evarianceprime}
  \BV \hat{F}=\int f(x)^2 e^{-\lambda(\{y:x\in[\delta_y]\})}\,\lambda(dx).
\end{displaymath}

\begin{remark}
  An unbiased estimator of the variance $\BV \hat{F}$ can be constructed as
  \begin{displaymath}
    \hat{V}:=   \int f(x)^2 H_x(\eta)\,\eta(dx)= \int f(x)^2 \, \partial\eta(dx).
  \end{displaymath}
  Indeed, the Mecke equation and Lemma~\ref{lemma:minus-point} yield
  that its expectation equals
  \begin{displaymath}
    \BE \hat{V}=    \BE \int f(x)^2 H_x(\eta+\delta_x)\,\lambda(dx)
    =\BE \int f(x)^2 H_x(\eta)\,\lambda(dx)=\BV \hat{F}. 
  \end{displaymath}
  Using the  Mecke equation and assuming $f\in L^4(\lambda)$, we obtain
  \begin{align*}
    \BE \hat{V}^2=\int f(x)^4\BE  H_x(\eta) \,\lambda(dx)
    +\int f(x)^2f(y)^2\BE\big[H_x(\eta+\delta_y)H_y(\eta+\delta_x)\big] \,\lambda^2(d(x,y)).
  \end{align*}
  If the prime property holds, then the second term on the right-hand side simplifies to
  \begin{align*}
    \int f(x)^2f(y)^2H_x(\delta_y)H_y(\delta_x)
    \BE \big[H_x(\eta)H_y(\eta)\big] \,\lambda^2(d(x,y)).
  \end{align*}
  It is not hard to see (using similar calculations as in the
  Supplement \cite{las:mol24s}) that
  \begin{align*}
    \BE \big[ H_x(\eta)H_y(\eta)\big]
    =\exp\big[-\lambda\big(\{z:H_x(z)=0\}\cup\{z:H_y(z)=0\}\big)\big].
  \end{align*}
  This yields a formula for the variance of $\hat{V}$.
\end{remark}

\begin{remark}
  By the Cauchy--Schwarz inequality,
  \begin{align*}
    \big(\BV \hat{F}\,\big)^2
    \leq \int f(x)^4\,\lambda(dx) \int(\BE H_x(\eta))^2\,\lambda(dx),
  \end{align*}
  with equality if and only if 
  $f(x)=c\BE H_x(\eta)$ for some $c\in\R$ and $\lambda$-a.e.\ $x$.
\end{remark}

\begin{example}
  \label{ex:f=one}
  If $\lambda$ is finite and $f$ identically equals one (so that
  $f\in L^1(\lambda)\cap L^2(\lambda)$), then $F^*=\eta(\BX)$ and
  $\hat{F}=\lambda([\eta])+(\partial\eta)(\BX)$.
  Since $\hat{F}$ is unbiased,
  \begin{math}
    \BE \lambda\big([\eta]^c\big)=\BE\card(\partial\eta). 
  \end{math}
  Moreover,
  \begin{displaymath}
    \BV \hat{F}=\BV\big[\lambda([\eta]^c)-\card(\partial\eta)\big]
    =\BE \int H_x(\eta)\lambda(dx)
    =\BE \lambda([\eta]^c). 
  \end{displaymath}
  Consider, in particular, Example~\ref{ex:chull-oracle} with
  $\lambda$ given as the Lebesgue measure $V_d$ restricted to a convex
  body $K\subset\R^d$.  Then $\BV \hat{F}=\BE V_d(K\setminus [\eta])$,
  as derived in \cite[Theorem~3.2]{BaldinReiss16}. The unbiasedness of
  $F$ implies that $\BE V_d(K\setminus [\eta])$ equals
  $\BE \card(\partial\eta)$, where $\card(\partial\eta)$ denotes the
  cardinality of $\partial\eta$, noting that $\eta$ does a.s.\ not
  have multiple points. This is the well-known Poisson version of  
  Efron's identity for random polytopes, see \cite[Theorem~2]{MR3293497}.
  In difference to \cite{BaldinReiss16},
  our setting includes nondiffuse intensity measures. The case of an
  infinite intensity measure is the subject of the following example.
\end{example}

\begin{example}
  \label{ex:poisson-polytope}
  Let $\lambda$ be the measure on $\R^d\setminus\{0\}$ with density
  $\|x\|^{-\alpha-d}$ with $\alpha>0$. The intensity of $\eta$ has a
  pole at the origin and so the total number of points in $\eta$ is
  infinite. Let $\partial\eta$ be the counting measure giving unit
  weights to each vertex of the convex hull of the support of
  $\eta$. It is well known that $\conv(\supp\eta)$ is a convex
  polyhedron with a finite number of vertices, see
  \cite[Corollary~4.2]{kab:mar:tem:19}. Hence, $\partial\eta$ is
  a.s. finite. In this case, $[\eta]$ is the convex hull of $\eta$
  with the extreme points (vertices) excluded. Let
  $f_\eps(x):=\I\{\|x\|\geq \eps\}$, so that $f_\eps\in L^1(\lambda)\cap L^2(\lambda)$. 
For this function, the  unbiasedness of $\hat{F}$ yields that
  \begin{displaymath}
    \BE \int f_\eps(z)H_z(\eta-\delta_z)\eta(dz)
    =\BE \int f_\eps(z) H_z(\eta)\lambda(dz).
  \end{displaymath}
  Letting $\eps\downarrow 0$ yields that $\BE \lambda([\eta]^c)=\BE
  \card(\partial\eta)$.  
\end{example}

\begin{example}
  \label{ex:two-dim-coordinatewise}
  Let $\BX=[0,1]^2$ and let $\lambda=t V_2$, where $V_2$ is the
  Lebesgue measure on $\BX$ and $t>0$.
  For $\mu\in\bN$ let $\partial\mu=0$ if $\mu=0$.  Otherwise, $\mu$
  contains points $(x_1,y_1)$ and $(x_2,y_2)$ such that $x_1$ is the
  smallest $x$-coordinate of all points from $\mu$ and $y_2$ is the
  smallest $y$-coordinate of all points from $\mu$. Let $\partial\mu$
  be the measure supported at $(x_1,y_1)$ and $(x_2,y_2)$ with the
  multiplicities inherited from $\mu$ if these two points are
  different and if $(x_1,y_1)=(x_2,y_2)$, let $\partial\mu$ be
  supported at this common point with the multiplicity inherited from
  $\mu$.
  For $f\equiv1$, we get
  $\hat{F}=\lambda([\eta])+\card(\partial\eta)$. The unbiasedness of
  $\hat{F}$ yields that
  \begin{displaymath}
    \BE\card(\partial\eta)=\BE\lambda\big([\eta]^c\big)
    =2\big(1-e^{-t}\big)-\frac{1}{t}\big(1-e^{-t}\big)^2.
  \end{displaymath}
\end{example}

Results on moments of Kabanov--Skorohod integrals from
\cite{Privault12} yield formulas for higher order moments of
$\hat{F}-F$, under suitable moment assumptions.  For example,
in view of 
Lemma~\ref{lemma:cyclic-bis}, the
recursive formula \cite[p.~968]{Privault12} yields that
\begin{displaymath}
  \BE\big[(\hat{F}-F)^{n+1}\big]
  =\sum_{k=0}^{n-1} \binom{n}{k}
  \int f(z)^{n-k+1}\BE\big[H_z(\eta)\deltaKS(fH)^k\big]\lambda(dz).
\end{displaymath}
For instance, 
\begin{displaymath}
  \BE\big[\deltaKS(fH)^3\big]=\int f(z)^3\BE H_z(\eta)\lambda(dz)
  +2\int f(z)^2 f(y) \BE\big[\big(D_yH_z(\eta)\big)H_y(\eta)\big]\lambda^2(d(y,z)).
\end{displaymath}

\section{Higher order conditional U-statistics}
\label{secestimatorbi}

We let $\eta$ and $\partial$ be as in Section~\ref{secestimator}.
We take a symmetric function $f\in L^1(\lambda^k)\cap L^2(\lambda^k)$
of $k$ arguments and would like to estimate
\begin{align*}
  F^{(k)}:=\int f(x_1,\dots,x_k)\,\lambda^k(d(x_1,\dots,x_k)).
\end{align*}
Given $n\in\N$ we denote by $\bx=(x_1,\dots,x_n)$ a generic element of
$\BX^n$.  The value of $n$ will always be clear from the context.  We
also write $\delta_{\bx}:=\delta_{x_1}+\cdots+\delta_{x_n}$.

By the multivariate Mecke formula we have
$F^{(k)}=\BE \int f\,d\eta^{(k)}$, where the integration is taken with
respect to the $k$-th factorial measure $\eta^{(k)}$ of $\eta$, see
\cite[Chapter~4]{LastPenrose17}.  Therefore, we define the Poisson
hull estimator as the conditional expectation
\begin{equation}
  \label{eq::hat-F-k}
  \hat{F}^{(k)}:=\BE\bigg[\int
    f(\bx)\,\eta^{(k)}(d\bx)\Bigm\vert \partial\eta\bigg].
\end{equation}
We can use Theorem~\ref{tstopping1} or Proposition~\ref{prop:infinite}
and a similar reasoning as in Lemma~\ref{lemma:hat-F-1}, to arrive at
\begin{align}
  \label{e078}
  \hat{F}^{(k)}
  =\sum_{i=0}^k \binom{k}{i}\iint f(\bx,\by)
  \bH_{\bx}(\eta)
  H_{\by}(\eta)
  \lambda^i(d\bx)\eta^{(k-i)}(d\by),
\end{align}
where 
\begin{align*}
  \bH_{\bx}(\eta):= \prod_{j=1}^n \bH_{x_j}(\eta),\quad
  H_{\bx}(\eta):=\prod_{j=1}^n H_{x_j}(\eta),\quad
  \bx\in\BX^n.
\end{align*}

For $m\in\{1,\ldots,k\}$, define a symmetric function
$f_m\in L^1(\lambda^m)$ by
\begin{align}\label{e5.7}
  f_m(\bx):=\binom{k}{m}\int f(\bx,\by)\,
  \lambda^{k-m}(d\by),\quad \bx\in\BX^m,
\end{align}
where $\int fd\lambda^0:=f$. We shall assume that these functions
$f_m$ are all square integrable.  It is well known (see
\cite{reit:sch13} or \cite[Proposition~12.12]{LastPenrose17}) that
\begin{align}\label{e5.67}
  \BV \int f\,d\eta^{(k)} =\sum^k_{m=1}  m!\int f_m^2\, d\lambda^m.
\end{align}

\begin{theorem}\label{tvariancemult} 
  Suppose $f\in L^1(\lambda^k)$ is symmetric and such that for each
  $m\in\{1,\ldots,k\}$ the function $f_m$ defined by \eqref{e5.7} is
  in $L^2(\lambda^m)$. Then $\hat{F}^{(k)}$ is square integrable and
  \begin{align}
    \label{multvariance}
    &\BV \hat{F}^{(k)}\\ \notag
    &=\sum^k_{m=1} \binom{k}{m}^2 m!
    \iint f(\bx,\by)f(\bx,\bz)
    \BE\left[\prod^m_{i=1}H_{x_i}(\eta+\delta_{\bx}+\delta_{\by}+\delta_{\bz})\right]\,
    \lambda^{2k-2m}(d(\by,\bz))\,\lambda^{m}(d\bx).
  \end{align}
\end{theorem}

The proof of Theorem~\ref{tvariancemult} is given in the supplement to
this paper.  If $\partial$ is the identity map, then $H\equiv 1$ and
\eqref{multvariance} turns into the variance formula \eqref{e5.67}.
If $k=1$, \eqref{multvariance} simplifies to \eqref{e3.47}. For
further illustration we state the case $k=2$.

\begin{corollary}
  \label{tvariancebi} 
  Suppose $f\in L^1(\lambda^2)\cap L^2(\lambda^2)$ is symmetric and
  satisfies
  \begin{align*} 
    \int \bigg(\int f(x,y)\,\lambda(dy)\bigg)^2\,\lambda(dx)<\infty.
  \end{align*}
  Then $\hat{F}^{(2)}$ is square integrable and
  \begin{align*}
    \BV \hat{F}^{(2)}
    &=2\int f(x,y)^2\BE\big[H_x(\eta+\delta_y)H_y(\eta+\delta_x)\big]\,\lambda^2(d(x,y))\\
    &\qquad +4\int f(x,y)f(x,z)\BE\big[H_x(\eta+\delta_y+\delta_z)\big]\,\lambda^3(d(x,y,z)).
  \end{align*}
\end{corollary}

Assuming that the generator satisfies the prime property,
$H_{x_i}(\eta+\delta_{\bx}+\delta_{\by}+\delta_{\bz})$ factorises into
a product, and so the variances of higher order conditional
symmetric statistics are given by
\begin{displaymath}
  \BV \hat{F}^{(k)} =\sum^k_{m=1}  m!\int
  \tilde{f}_m^2(\bx)\BE \big[H_{\bx}(\eta)\big] H_{\bx}(\delta_{\bx})
  \, \lambda^m(d\bx), \quad k\in\N,
\end{displaymath}
where
\begin{displaymath}
  \tilde{f}_m(\bx):=\binom{k}{m}\int f(\bx,\by)
  H_{\bx}(\delta_{\by})\,
  \lambda^{k-m}(d\by),\quad \bx\in\BX^m.
\end{displaymath}

\begin{example}
  Consider higher order integrals in the setting of
  Example~\ref{ex:chull-oracle}.
  If $f\equiv1$ and $\lambda=V_d$ is the Lebesgue measure on $K$, then
  $F^{(2)}=V_d(K)^2$.  Recall that $[\eta]$ is the convex hull $Z$ of
  $\supp\eta$ with the vertices removed.  The Poisson hull
  estimator \label{eq:F^2} of $V_d(K)^2$ becomes
  \begin{displaymath}
    \hat{F}^{(2)} = V_d(Z)^2 +2 V_d(Z) M+M(M-1),
  \end{displaymath}
  where $M$ is the number of vertices of $Z$.  The Poisson hull
  estimator \eqref{e078} of $V_d(K)^k$ is given by
  \begin{displaymath}
    \hat{F}^{(k)}
    =\sum_{i=0}^k \binom{k}{i} V_d(Z)^i\frac{M!}{(M-k+i)!}. 
  \end{displaymath}
  The unbiasedness of the estimator $\hat{F}^{(k)}$ means that
  \begin{displaymath}
    \sum_{i=0}^k \binom{k}{i} \BE \Big[V_d(Z)^i \frac{M!}{(M-k+i)!}\Big] 
    = V_d(K)^k,
  \end{displaymath}
  which is apparently a new result concerning the joint moments of
  the volume of the convex hull and the number of vertices.

  Let $f(x,y)=\|x-y\|^j$ for $j>-d$. If $\lambda$ is the
  Lebesgue measure on $K$, then
  \begin{displaymath}
    F^{(2)}=\int_K\|x-y\|^j\, \lambda^2(d(x,y))
    =\frac{2}{(d+j)(d+j+1)} I_{d+j+1}(K),
  \end{displaymath}
  where $I_n(K)$ is the chord power integral of order $n$, see
  \cite[p.~364]{sch:weil08}. The unbiased estimator of $F^{(2)}$ is 
  \begin{displaymath}
    \hat{F}^{(2)}=\frac{2}{(d+j)(d+j+1)} I_{d+j+1}(Z)
    +2 \sum_{y\in\partial\eta} \int_Z\|x-y\|^j \,dx
    +\sum_{x,y\in\partial\eta, x\neq y} \|x-y\|^j.
  \end{displaymath}
\end{example}

\section{Normal approximation}
\label{sec:centr-limit-theor}

Consider the estimator $\hat{F}$ from Section~\ref{secestimator} with
$f\in L^1(\lambda)\cap L^2(\lambda)$.  In order to formulate a central
limit theorem, consider Poisson processes $\eta_t$ with intensity
measure $t\lambda$ for $t>0$ and define
$\hat{F}_t:=\hat{F}(\eta_t)$. By \eqref{e3.47},
\begin{align*}
  \sigma_t^2:=\BV \hat{F}_t=t\int f(z)^2 \BE H_z(\eta_t)\,\lambda(dz). 
\end{align*}
Our aim is to derive the limit distribution of
$(\hat{F}_t-tF)/\sigma_t$ as $t\to\infty$. For this, additional
conditions are necessary, for instance, the central limit theorem does
not hold if the cardinality of $\partial\eta_t$ does not grow to
infinity, like in Example~\ref{ex:two-dim-coordinatewise}.

Along with assuming that $f\in L^1(\lambda)\cap L^2(\lambda)$, we 
additionally impose the following integrability conditions:
\begin{align}
  &\int f(y)^2 \BE \big(D_xH_y(\eta)\big)^2
    \,\lambda(d(x,y))<\infty, \label{eq:int2}\\
  &\int f(y)^2 \BE\big(D_{x,z}^2H_y(\eta)\big)^2
    \,\lambda(d(x,y))<\infty, \label{eq:int3}\\
  &\int f(y)^2 \BE \big(D_{x,z,w}^3 H_y(\eta)\big)^2
    \,\lambda(d(y,z,w))<\infty, \;\; \lambda\text{-a.e.}\; x, \label{eq:int4}
\end{align}
where the expressions for successive differences can be found at
\eqref{eq:16}. Since $H_y(\eta)$ takes values $0$ or $1$, these
conditions follow from the square integrability of $f$ if the
intensity measure $\lambda$ is finite. 

We denote by $d_W(X,Y)$ the {\em Wasserstein distance} between the
distributions of two random variables $X$ and $Y$, see, e.g.,
\cite{LastPenrose17}. The following result is proved in the supplement
as an application of a central limit theorem for Kabanov--Skorohod
integrals from \cite{las:mol:sch20}. We maintain the
notation 
used in the cited paper. Denote by $N$ the standard Gaussian random
variable.

\begin{theorem}
  \label{theorem:W-distance}
  Assume that $f\in L^1(\lambda)\cap L^2(\lambda)$ and that the
  conditions \eqref{eq:int2}, \eqref{eq:int3} and \eqref{eq:int4} are
  satisfied. Let $t>0$ be such that $\sigma_t^2>0$. Then
  \begin{equation}
    \label{eq:10}
    d_W\Big(\frac{\hat{F}_t-tF}{\sigma_t},N\Big)
    \leq T_1(t)+T_3(t)+T_4(t)+T_5(t),
  \end{equation}
  where 
  \begin{align*}
    T_1(t)&:=t^{3/2}\sigma_t^{-2}\bigg(\int f(x)^2f(y)^2
            \BE\Big[ H_x(\eta_t)H_y(\eta_t)\bH_x(\eta_t+\delta_z)
            \bH_y(\eta_t+\delta_z)\Big] \,\lambda^3(d(x,y,z))\bigg)^{1/2},\\
    T_3(t)&:=t\sigma_t^{-3}\int |f(x)|^3\BE H_x(\eta_t)\,\lambda(dx),\\
    T_4(t)&:=t^2\sigma_t^{-3} \int \Big(3|f(x)|f(y)^2+2f(x)^2|f(y)|\Big)
            \BE\Big[H_y(\eta_t)\bH_y(\eta_t+\delta_x)\Big]\,\lambda^2(d(x,y)),\\
    T_5(t)&:=8 t^3\sigma_t^{-3}\int |f(x)f(y)f(z)|
            \BE\Big[H_z(\eta_t)\bH_z(\eta_t+\delta_y)
    \bH_y(\eta_t+\delta_x)\Big]\,\lambda^3(d(x,y,z)).
  \end{align*}
\end{theorem}

\begin{remark}
  \label{rem:CLT-general-functionals}
  The rate given in Theorem~\ref{theorem:W-distance} applies also to
  the Kabanov--Skorohod integral of any function given by the product
  of a function of $x$ and a function $H_x(\eta)$ with values
  $\{0,1\}$ which is decreasing and such that the corresponding
  function $\bH_x(\eta)=1-H_x(\eta)$ satisfies \eqref{eq:17}.
\end{remark}

In the prime setting the terms involved in the bound on the Wasserstein
distance and the required integrability conditions become
simpler and lead to the following corollary proved in the
supplementary material. Denote
\begin{equation}
  \label{eq:h_i}
  h_i(y):=\int |f(x)|^i\bH_y(\delta_x)\,\lambda(dx),\quad y\in\BX,\; i=0,1,2.
\end{equation}

\begin{corollary}
  \label{cor:normal-prime}
  Assume that $f\in L^1(\lambda)\cap L^2(\lambda)$ and
  \begin{equation}
    \label{eq:int-prime}
    \int f(y)^2 h_0(y)^2 \BE H_y(\eta) \,\lambda(dy)<\infty.
  \end{equation}
  If the prime property \eqref{eq:6} is satisfied, then \eqref{eq:10}
  holds with $T_3(t)$ as in Theorem~\ref{theorem:W-distance} and 
  \begin{align*}
    T_1(t)&=t^{3/2}\sigma_t^{-2} \bigg(
            \int f(x)^2f(y)^2\BE\big[H_x(\eta_t) H_y(\eta_t)\big]\bH_x(\delta_z)
    \bH_y(\delta_z)\,\lambda^3(d(x,y,z))\bigg)^{1/2},\\
    T_4(t)&=t^2\sigma_t^{-3}\int \Big(3 h_1(y)f(y)^2+2h_2(y)|f(y)|\Big)
            \BE H_y(\eta_t)\,\lambda(dy),\\
    T_5(t)&=8t^3\sigma_t^{-3}\int |f(z)|h_1(z)^2\BE H_z(\eta_t)
            \,\lambda(dz).
  \end{align*}
\end{corollary}

\begin{remark}
  \label{rem:Kolm}
  Assume that $\partial$ satisfies \eqref{eq:nontrivial}, that the
  prime property holds, and recall from Remark~\ref{rem:prime} the
  relationship between the function $H_x$ and a strict partial order.
  Then a result from \cite[Proposition~7.1]{las:mol:sch20} yields not
  only simple expressions (identical to those given in
  Corollary~\ref{cor:normal-prime} above) for the Wasserstein but also
  for the Kolmogorov distance, that is, the uniform distance between the
  cumulative distribution functions of $\sigma_t^{-1}(\hat{F}_t-tF)$
  and $N$.
\end{remark}

\section{Poisson processes on function spaces}
\label{sec:poiss-proc-funct}

\subsection{General setting}
\label{sec:general-setting}

In this section we assume that $\BX$ is a subset of the space of
measurable functions $x:\BS\to[-\infty,\infty)$ on a locally compact
Polish space $\BS$ such that $x(s)\geq \varpi(s)$ for all $s\in\BS$,
where $\varpi\colon\BS\to[-\infty,\infty)$ is a function which is
continuous on the set $\{s\in\BS:\varpi(s)>-\infty\}$ and the latter
set is assumed to be open (therefore, $\varpi$ is lower
semicontinuous). In many cases it is possible to let
$\varpi(s)=-\infty$ for all $s$.  We equip $\BX$ with the smallest
$\sigma$-field $\cX$, such that the mappings $x\mapsto x(r)$ are
$\cX$-measurable for each $r\in\BS$.
The space $\BX$ and the localising ring
$\cX_0\subset\cX$ are assumed to have the following properties.

\begin{enumerate}
\item[(F0)] $(\BX,\cX)$ is a Borel space.
\item [(F1)] For each $x\in\BX$, 
  the set $\{r\in\BS:x(r)>\varpi(r)\}$
  is 
  open, and $x$ is continuous on it.
\item [(F2)] If, for $x,x'\in \BX$ and an open set $U\subset\BS$, the
  set $\{r\in U:x(r)>\varpi(r),x'(r)>\varpi(r)\}$ is nonempty and
  the functions $x$ and $x'$ coincide on this set, then
  $x=x'$.
\item [(F3)] If $\varpi\in\BX$, then $\{\varpi\}\in\cX_0$. Moreover,
  for each compact set $K\subset \BS$, each $\eps>0$ and each $c\in\R$
  we have
  \begin{align*}
    \big\{x\in \BX: \text{$x(r)\geq \varpi_{c,\eps}(r)$
    for some $r\in K$}\big\}\in\cX_0,
  \end{align*}
  where
  $\varpi_{c,\eps}(r):=\I\{\varpi(r)=-\infty\}c
  +\I\{\varpi(r)>-\infty\}(\varpi(r)+\eps)$.  Conversely, any set 
  $A\in\cX_0$ with $\varpi\notin A$ is contained in one of these sets.
\end{enumerate}

For a counting measure $\mu$ on the functional space $\BX$, write
$\sup \mu$ for the function obtained as the pointwise supremum of
functions $x$ for $x\in\mu$, and let $\sup\mu=\varpi$ if $\mu=0$.  For
$x\in\BX$, denote $\mu_{-x}:=\mu-\mu(\{x\})\delta_x$.  Define the
generator of $\mu\in\bN$ by
\begin{equation}
  \label{eq:24}
  \partial\mu:=\sum_{x\in\mu} \mu(\{x\})
  \I\{\sup \mu\neq\sup \mu_{-x}\}
  \delta_{x},
\end{equation}
that is, the generator is the restriction of $\mu$ onto the set of
$x\in\mu$ corresponding to functions contributing to $\sup\mu$.  Note
that $\sup\mu=\sup\partial\mu$. The function 
$\varpi$ never belongs to $\partial\mu$. 

In the following, inequalities between functions are understood
pointwisely for all their arguments. The following
result is proved in the Supplement \cite{las:mol24s}. 

\begin{lemma}\label{lhullfunctions} 
  Let $\mu\in\bN$. The generator given by \eqref{eq:24} satisfies
  (H1)--(H4), is measurable, and the corresponding hull operator is
  given by
  \begin{equation}
    \label{eq:27h}
    [\mu]=\{x\in\BX: x\leq \sup\mu\}\setminus \supp\partial\mu.
  \end{equation}
\end{lemma}

Let $\lambda$ be a measure on $(\BX,\cX)$, which is finite on $\cX_0$.
Let $\eta$ be a Poisson process on $\BX$ with intensity measure
$\lambda$.  Below we shall prove the strong Markov property.  If
$\BP(\partial\eta(\BX)=\infty)>0$, that is, the generator is infinite
with a positive probability, we need the following additional
assumption.
\begin{enumerate}
\item[(F4)] For $\Pi_\lambda$-a.e.\ $\psi\in\bN$ and $\lambda$-a.e.\
  $x\in\BX$ the following is true. If $x\le \sup\psi$ then there exists
  a finite $\psi'\in\bN$ such that $\psi'\le \psi$ and $x\le \sup\psi'$.
\end{enumerate}
If $\lambda(\BX)<\infty$, then (F4) holds. It is easy to see that 
(F4) holds if the generator has the prime property. The following
result is proved in the Supplement \cite{las:mol24s}. 

\begin{lemma}\label{lMarkovfunctions}
  If $\lambda(\BX)<\infty$, then the strong Markov property
  \eqref{stopping3} holds. In the general case this remains true if
  condition (F4) holds.
\end{lemma}

\subsection{Parametric families of functions}
\label{sec:param-famil-funct}

In this subsection we assume that $\BX$ is a family of measurable
functions $g_{s,u}\colon\BS\to[-\infty,\infty)$ parametrised by
$(s,u)$ from the product space $\BS\times\R$. We let
$\varpi\colon\BS\to[-\infty,\infty)$ be as in Subsection
\ref{sec:general-setting}, that is, we assume that $g_{s,u}\ge \varpi$
for all $(s,u)\in\BS\times\R$.  To ensure condition (F0), we assume
that $g_{s_n,u_n}(r)\to g_{s,u}(r)$ for all $r\in\BS$ if and only if
$s_n\to s$ and $u_n\to u$ as $n\to\infty$. This is the case if the
parametrisation is continuous with respect to the pointwise
convergence of functions. Then the map $(s,u,r)\mapsto g_{s,u}(r)$ is
measurable. Condition (F1) is ensured by assuming that all functions
$g_{s,u}$ are continuous on the set $\{r\in\BS:g_{s,u}(r)>\varpi(r)\}$
and the latter set is open.  We assume that if
$g_{s,u}(r)=g_{s',u'}(r)$ for $r$ from a nonempty open set in $\BS$,
where the functions are different from $\varpi(r)$, then
$(s,u)=(s',u')$.  This implies (F2).  In the following we
shall tacitly identify $\BX$ with $\BS\times\R$.  In our current more
specific setting we reformulate assumption (F3) on $\cX_0$ as follows.
\begin{itemize}
\item [(F3$^\prime$)] If $\varpi\in\BX$, then
  $\{\varpi\}\in\cX_0$. Moreover, a set $A\in\cX$ with
  $\varpi\notin A$ is in $\cX_0$ if and only if there exist a compact
  set $K\subset\BS$, an $\eps>0$ and a $c\in\R$ such that
  \begin{align}\label{eint7}
    A\subset    \big\{(s,u)\in\BS\times\R:g_{s,u}(r)\geq \varpi_{c,\eps}(r)
    \;\text{for some}\; r\in K\big\}.
  \end{align}
\end{itemize}
Note that by our continuity assumptions on the functions $g_{s,u}$ the
set on the right-hand side of \eqref{eint7} is in $\cX$. By
(F3$^\prime$) this set belongs to $\cX_0$.  Finally we assume that
$g_{s,u}(r)$ is increasing in $u$ for each $s,r\in\BS$. This holds
in all our examples, where often $g_{s,u}=g_{s,0}+u$ for all
$s\in\BS$ and $u\in\R$.

Let $\phi\colon\BS\to[-\infty,\infty)$ be a measurable function such
that $\phi\geq\varpi$, the set $\{r\in\BS:\varphi(r)>\varpi(r)\}$ is
open and $\varphi$ is continuous on it.  Let $\bar\lambda$ be the
product of a locally finite measure $\nu$ on $\BS$ and a locally
finite diffuse measure $\theta$ on $\R$.  As in
Remark~\ref{rem:T-sufficient}, we work with the measure $\lambda$
obtained as the restriction of $\bar\lambda$ onto a Borel set in
$\BH\subset\BS\times\R$ defined as
\begin{displaymath}
  \BH:=\big\{(s,u)\in \BS\times\R: \varpi\ne g_{s,u}\leq\varphi\big\}.
\end{displaymath}
We refer to $\phi$ as the \emph{boundary function}.  In specific
examples we need to check that $\lambda$ is indeed locally finite
(that is, finite on $\cX_0$), see Remark~\ref{r:locfinite} and
Example~\ref{ex:Legendre}. If necessary, we impose (F4).  In a
statistical context the set $\BH$, equivalently, $\phi$, is regarded
as unknown. The goal is to recover information on $\phi$ by observing a
Poisson process $\eta$ on $\BS\times\R$ with intensity measure
$\lambda$, equivalently, a Poisson process on $\BX$ with intensity
measure being the pushforward of $\lambda$ under the map
$(s,u)\mapsto g_{s,u}$.

By the monotonicity of $g_{s,u}$ in $u$, the set
$\{u:\varpi\ne g_{s,u}\le \varphi\}$ is an interval for each
$s\in\BS$, and so
\begin{displaymath}
  \BH=\big\{(s,u)\in \BS\times\R: \varpi_*(s)\leq u\leq \phi^-(s)\big\},
\end{displaymath}
where 
\begin{align}
  \varpi_*(s)&:=\inf\big\{u: \varpi\ne g_{s,u}\big\},\quad s\in\BS, \notag\\
  \label{eq:15}
  \phi^-(s)&:=\sup\big\{u\in\R: g_{s,u}\leq \phi\big\},\quad s\in\BS.
\end{align}
If the set under the infimum (respectively, supremum) is empty, we let
$\varpi_*(s)=\infty$ (respectively, $\phi^-(s):=\varpi_*(s)$). Then
\begin{equation}
  \label{elambda1}
  \lambda(d(s,u))
  =\I\big\{\varpi_*(s)\leq u\leq \phi^-(s)\big\}\,\nu(ds)\,\theta(du). 
\end{equation}

Fix a function $f\in L^1(\lambda)$, and let
$F:=\int f\,d\lambda$ as before.  Then
\begin{align}
  \label{eq:20}
  F=\int_{\BS} \tilde{f}\big(s,\phi^-(s)\big)\,\nu(ds),
\end{align}
where
\begin{equation}
  \label{eq:tilde-f}
  \tilde{f}(s,u):=\int_{\varpi_*(s)}^u f(s,t)\, \theta(dt).
\end{equation}

We write $\sup\eta$ for the function obtained as the supremum of
$g_{s,u}$ for $(s,u)\in\eta$, and note that $(\sup\eta)^-$ (defined as
at \eqref{eq:15} with $\phi$ replaced by $\sup\eta$) is the supremum
of all functions $g_{s,u}\in[\eta]$, where the generator and the hull
operator are defined as in the previous subsection. By
Lemma~\ref{lhullfunctions}, the Poisson hull estimator \eqref{e6.23}
of $F$ becomes
\begin{align}
  \label{eq:26}
  \hat{F}=\int \tilde{f}\big(s,(\sup\eta)^-(s) \big)\,
  \nu(ds)+\int f(s,u)\, \partial\eta(d(s,u)).
\end{align}
The first term on the right-hand side is an empirical (plugin) version
of \eqref{eq:20}. The second summand in \eqref{eq:26} can be
interpreted as an additive correction term. Recall that $\partial\eta$
is a sufficient statistic for $\BH$ and so for $F$, see
Remark~\ref{rem:T-sufficient}.

\begin{remark}
  \label{rem:minimum}
  The definition of the generator \eqref{eq:24} can be amended by
  replacing the supremum with the infimum. In this case all above
  constructions apply with $\phi^-$ replaced by
  $\phi^+(s):=\inf\{u\in\R: g_{s,u}\geq \phi\}$ and reversing the
  inequalities in 
  \eqref{elambda1}. The first term
  in $\hat{F}$ from \eqref{eq:26} overestimates $F$
  from~\eqref{eq:20}, and the hull estimator is obtained by changing
  the sign in front of the second term in \eqref{eq:26}.
\end{remark}

It is easy to see that $(\sup\eta)^-=\sup\eta$ if the condition of the
following lemma holds and $g_{s,u}(s)=u$ for all $s$ and $u$. 

\begin{lemma}
  \label{lemma:domination}
  Assume that $\BX$ consists of functions $g_{s,u}$ such that
  $g_{s,u}(r)\leq g_{s,u}(s)$ for all $s,r\in\BS$, $u\in\R$, and
  $g_{s,u}\leq g_{s',u'}$ for $s,s'\in\BS$ and $u,u'\in\R$ if and only if
  $g_{s,u}(s)\leq g_{s',u'}(s)$. Then the generator has the prime
  property.
\end{lemma}
\begin{proof}
  Let $\mu\in\bN$ and $x=g_{s,u}\in[\mu]$, that is
  $\partial(\mu+\delta_{x})=\partial\mu$.  By
  Lemma~\ref{lemma:prime-c} we need to show that $x\in[\delta_{y}]$
  for some $y\in\mu$.  If $\mu=0$, then $[\mu]=\emptyset$, and there
  is nothing to prove.  Assume that $\mu\ne 0$.  We have
  $x\notin\partial\mu$ and $x\leq \sup \mu$.  If $x=\varpi$, then
  $x\leq y$ for each $y\in\mu$, so that
  $\partial(\delta_y+\delta_x)=\partial(\delta_y)$.  Otherwise,
  $g_{s,u}(s)>\varpi(s)$. By (F3$^\prime$) with $K=\{s\}$, there is at most a
  finite set of functions $y\in\mu$ such that $y(s)\geq
  g_{s,u}(s)$. Thus,
  \begin{math}
    g_{s,u}(s)\leq (\sup\mu)(s)=y(s)
  \end{math}
  for some $y\in\mu$. By assumption we have $x\le y$,
  hence, $x\in[\delta_y]$.
\end{proof}

\begin{example}
  Assume that $\varpi=\varpi_*\equiv 0$, $\theta$ is the Lebesgue
  measure, and let $f(s,u):=\Phi'(u)$, where
  $\Phi'\colon [0,\infty) \to\R $ is locally integrable.  Then
  $\tilde{f}(s,u)=\Phi(u)$, where $\Phi(u):=\int_0^u\Phi'(t)\,dt$,
  $u\in[0,\infty)$.  Fix a boundary function $\phi$ and consider
  $\lambda$ from \eqref{elambda1}. Then
  \begin{align*}
    F=\iint f(s,u)\,du\,\nu(ds)=\int \Phi \big(\phi^-(s) \big)\,\nu(ds),
  \end{align*}
  provided that $f\in L^1(\lambda)$. (The latter holds, for instance,
  if $\Phi'\ge 0$ and $\int \Phi(\varphi^-(s))\,\nu(ds)<\infty$.)  The
  Poisson hull estimator \eqref{eq:26} is given by
  \begin{align*}
    \hat{F}=\int \Phi \big ((\sup \eta)^-(s) \big)\,\nu(ds)
    +\int f(s,u)\, \partial\eta(d(s,u)).
  \end{align*}
  If, for example, $\Phi'(u):=pu^{p-1}$ with $p>0$, we estimate
  $F=\int (\phi^-(s))^p\,\nu(ds)$. If $\BS=[0,1]$ and $\nu$ the
  Lebesgue measure, this estimator was studied in \cite{ReissSelk17}
  (with $\Phi$ chosen to be the identity) and \cite{ReissWahl19}.  In
  fact, the cited papers dealt with an infimum instead of a supremum,
  which amounts to changing the sign of the compensating term, see
  Remark~\ref{rem:minimum}.
\end{example}

\begin{remark}
  If $\BX$ is the space of constant functions on $\BS$, then $\BX$ can
  be identified with $\R$ and we can take $\varpi\equiv -\infty$. A
  boundary function is then just a number $a\in\R$. Up to a sign this
  is Example~\ref{ex:minimum}.
\end{remark}

\begin{remark}\label{r:locfinite}
  We have assumed that $\lambda$ is locally finite, i.e., that
  $\lambda$ is finite on sets of the form \eqref{eint7}.
  In some cases this assumption can be verified as follows.
  Assume that $g_{s,u}(r)\leq u$ for all $r\in\BS$ and all $s,u$.  Set
  $a:=\inf\{\varpi_{c,\eps}(r):r\in K\}>\inf\varpi$. If $\theta$ is
  absolutely continuous with density bounded by $c$, then for the set
  $A$ from the left-hand side of \eqref{eint7} we have
  \begin{align*}
    \lambda(A)\leq \int\I \big\{a\leq u\leq \phi^-(s) \big\}\,\lambda(d(s,u))
    \leq c\int \big (\phi^-(s)-a \big)_+\,\nu(ds),
  \end{align*}
  where $u_+$ denotes the positive part of $u\in\R$.
\end{remark}

The following provides an example of the family $\{g_{s,u}\}$. The
next subsections will be devoted to an extensive discussion of other
examples.

\begin{example}\label{ex:Legendre}
  Let $\BS=\R$, $\varpi\equiv 0$, and let $g_{s,u}(r)=(sr+u)_+$ for
  $s,u\in\R$. If $\phi\colon\R\to\R_+$ is a convex function, then
  \begin{align*}
    \phi^-(s)=\sup\big\{u\in\R: sr+u\leq \phi(r)\;\text{for all}\; r\in\R\big\}
    =-\sup_{r\in\R} \big (sr-\phi(r) \big)=-\phi^o(s),
  \end{align*}
  where $\phi^o$ is the Legendre transform of $\phi$. Furthermore,
  $\varpi_*(0)=0$ and $\varpi_*(s)=-\infty$ for $s\neq0$. Let $\nu$ be
  a locally finite measure on $\R$, and let $\theta$ be Lebesgue
  measure on $\R$. Assuming $\nu(\{0\})=0$, we have
  \begin{align*}
    \lambda(d(s,u))=\I\big\{u\le -\phi^o(s)\big\}\,\nu(ds)\,du.
  \end{align*}
  Conditions (F1) and (F2) evidently hold, (F0) holds due to
  continuous parametrisation of $g_{s,u}$.  
  We now check that $\lambda$ is locally finite.
  Let $K=[-b,c]$ with $b,c>0$. For $a\ge 0$,
  \begin{align*}
    \lambda\big(\big\{(s,u):
    &(sr+u)_+> a \; \text{for some}\; r\in [-b,c]\big\}\big)\\
    &=\lambda \big (\big\{(s,u):sc+u> a, s\geq 0 \big\}\big)
    +\lambda \big (\big\{(s,u):-sb+u> a, s< 0 \big\}\big)\\
    &=\int_0^\infty \big(\phi^-(s)-a+sc\big)_+\,\nu(ds)
    +\int_{-\infty}^0 \big(\phi^-(s)-a-sb\big)_+\,\nu(ds).
  \end{align*}
  Since $\phi^-(s)-a+sc\geq 0$ if and only if $\phi(r)\ge s(r-c)+a$
  for all $r\in\R$, the first integrand does not vanish at most for
  $s$ from a compact set. With a similar argument applied to the
  second integrand, we see that the local finiteness of $\lambda$
  follows from the corresponding property of Lebesgue measure.
  
  We now confirm (F4). 
  Take a function $g_{s_0,u_0}$ which is dominated by $\sup\eta$.
  Assume that $s_0>0$. Since $\lambda$ is diffuse, we can assume
  $g_{s_0,u_0}\notin\eta$. There must be a $g_{s,u}\in \eta$ with
  $s>s_0$.  (Otherwise we cannot have $g_{s_0,u_0}\le \sup\eta$.)  The
  graphs of $g_{s_0,u_0}$ and $g_{s,u}$ intersect at some point with
  first coordinate $v_1$, say.  On $[v_1,\infty)$ the function
  $g_{s_0,u_0}$ is dominated by $g_{s,u}$.  Let $v_0:=-u_0/s_0$ be the
  largest zero of $g_{s_0,u_0}$.  On the interval $[v_0,v_1]$ the
  function $g_{s_0,u_0}$ is dominated by the supremum of functions
  $r\mapsto sr+u$ which exceed the value $0$ somewhere on $[v_0,v_1]$
  and for which $(s,u)\in\eta$.  Since we have already shown that the
  set of $(s,u)$ such that $sr+u> 0$ for some $r\in [v_0,v_1]$ has
  finite $\lambda$-measure, there is only a finite number of such
  functions.

  If $\int (\phi^-(s))_+\,\nu(ds)<\infty$,
  that is, if the negative part of $\phi^o$ is $\nu$-integrable, then
  we can take $f(s,u)=\I\{u\geq 0\}$, so that $F$ is the integral of
  the negative part of $\phi^o(s)$ and obtain its estimator from
  \eqref{eq:26}.
\end{example}

\subsection{Approximation of H\"older functions}
\label{sec:prime-sett-funct}

We now specify the setting of Subsection~\ref{sec:param-famil-funct}.
Assume that $\BS=\R^d$ with the Euclidean norm $\|\cdot\|$ and let
$\varpi$ be identically equal to $-\infty$.  Fix some $R\ge
0$. For $(s,u)\in\BS\times\R$, define the function
$g_{s,u}\colon\BS\to\R$ by
\begin{equation}
  \label{eq:18}
  g_{s,u}(r):=-R\|s-r\|^\beta+u, \quad r\in \R^d.
\end{equation}
Then $g_{s,u}=g_{s,0}+u$ for all $s$ and $u$.
Conditions (F0)-(F2) hold and (F3$^\prime$) is assumed to be in force.
It means that $\cX_0$ is the smallest ring containing
the sets $K\times[c,\infty)$ for a compact set $K\subset\BS$ and $c\in\R$. 
The following result implies also the validity of (F4).

\begin{lemma}
  \label{lemma:f-prime}
  For the functions given by \eqref{eq:18}, the generator 
  given by \eqref{eq:24} has the prime property.
\end{lemma}
\begin{proof}
  By Lemma~\ref{lemma:domination}, it suffices to show that
  $g_{s,u}\leq g_{s',u'}$ if and only if $g_{s,u}(s)=u\leq g_{s',u'}(s)$. The
  only if part is obvious.  Let $s,s'\in\BS$ and $u,u'\in\R$ and
  assume that $g_{s,u}(r)> g_{s',u'}(r)$ for some $r\in\BS$, that is,
  \begin{displaymath}
    u-R\|s-r\|^\beta> u'-R\|s'-r\|^\beta.
  \end{displaymath}
  By the triangle inequality and subadditivity of the
  function $t\mapsto t^\beta$,
  \begin{displaymath}
    \|s'-r\|^\beta-\|s-r\|^\beta\leq \|s'-s\|^\beta.
  \end{displaymath}
  Hence,
  \begin{displaymath}
    u>u'-R\|s'-r\|^\beta+ R\|s-r\|^\beta
    \geq u'-R \|s'-s\|^\beta=g_{s',u'}(s),
  \end{displaymath}
  which is a contradiction. 
\end{proof}

Consider a function $\phi\colon\BS\to[0,\infty)$,
satisfying the H\"older condition
\begin{align}
  \label{eLipschitz}
  \big|\phi(s)-\phi(s') \big|\le R' \|s-s'\|^\beta,\quad s,s'\in\BS,
\end{align}
for some $R'\in[0,R]$ and $\beta\in(0,1]$.
Assumption~\eqref{eLipschitz} yields that
$g_{s,\phi(s)}=g_{s,0}+\phi(s)\le \phi$, so that $\phi^-=\phi$. Indeed, if
$s,r\in\BS$ and $u\in\R$ satisfy $u-R\|s-r\|^\beta>\phi(r)$ and
$u\le\phi(s)$ then $\phi(s)-R\|s-r\|^\beta>\phi(r)$, contradicting
\eqref{eLipschitz} (and $R'\le R$).
Since $\phi^-=\phi$,
\begin{align*}
  \BH:=\big\{(s,u)\in\R^d\times\R: g_{s,u}\le \varphi \big\}
  =\big\{(s,u)\in\R^d\times\R: u\le \varphi(s) \big\}.
\end{align*}
If $B\subset \R^d\times\R$ is a Borel set such that
$B\cap\BH$ is bounded in $\R^d\times\R$, then
$\{g_{s,u}:(s,u)\in B\}\in\cX_0$.

Let $\lambda$ be the Lebesgue measure on $\R^d\times\R$ restricted to
$\BH$, so that $\nu$ and $\theta$ are Lebesgue measures on $\R^d$
and $\R$, respectively.  Let $\eta$ be a Poisson process with intensity measure
$\lambda$, and fix a function
$f\in L^1(\lambda)\cap L^2(\lambda)$.  The Poisson hull estimator
\eqref{eq:26} for $F=\int fd\lambda=\int \tilde{f}(s,\phi(s))\,ds$
with $\tilde{f}$ defined at \eqref{eq:tilde-f} is given by
\begin{equation}
  \label{eq:estim-Hoelder}
  \hat{F}(\eta)=\int \tilde{f}\big (s,(\sup\eta)(s) \big)\,ds
  +\int \I \big\{u=(\sup\eta)(s) \big\}f(s,u)\, \eta(d(s,u)).
\end{equation}
Here we have used that $(\sup\eta)^-=\sup\eta$ and that
$(s,u)\in\partial\eta$ if and only if $u=(\sup\eta)(s)$.

In the following write $H_{s,u}(\eta)\equiv H_x(\eta)$ for $x=(s,u)$, let
$\kappa_d$ be the volume of the unit ball in $\R^d$, and denote
$u\vee v:=\max(u,v)$. The next result is proved in the supplement. 

\begin{lemma}
  \label{lemma:H-bound}
  For all $s\in\R^d$ and $u\geq 0$, we have 
  \begin{equation}
    \label{eq:3}
    \frac{\beta\kappa_d}{d+\beta}(2R)^{-d/\beta} u^{(d+\beta)/\beta}
    \leq -\log \BE H_{s,\phi(s)-u}(\eta) \leq 
    \frac{\beta\kappa_d}{d+\beta}(R-R')^{-d/\beta} u^{(d+\beta)/\beta},
  \end{equation}
  where the right-hand side is set to be infinite if $R=R'$. 
  Furthermore, for $r\in\R^d$ and $v\geq0$,
  \begin{equation}
    \label{eq:4}
    \frac{d\kappa_d}{d+\beta}(2R)^{-d/\beta} (u\vee
    v)^{(d+\beta)/\beta}\leq
    -\log \BE\Big[H_{s,\phi(s)-u}(\eta)H_{r,\phi(r)-v}(\eta)\Big].
  \end{equation}
\end{lemma}

For functions $h_1(u)$ and $h_2(u)$, $u\in\R$, we write $h_1\asymp
h_2$ as $u\to u_0$ (where $u_0$ may be infinite) if 
\begin{displaymath}
  0<c_1\leq \liminf_{u\to u_0} h_1(u)/h_2(u)
  \leq \limsup_{u\to u_0} h_1(u)/h_2(u)\leq c_2<\infty
\end{displaymath}
for some constants $c_1$ and $c_2$. The function $h_1$ is said to grow
at most polynomially if there exists a $p\geq 0$ such that
$h_1(u)u^{-p}$ is bounded on $[\eps,\infty)$ for some
$\eps>0$. The following result is a Tauberian style statement on the
asymptotic behaviour of integrals. Its straightforward proof is omitted. 

\begin{lemma}
  \label{lemma:Tauber}
  Let $h:(0,\infty)\to\R_+$ be a function which grows at most
  polynomially. Assume that $h(u)\asymp u^{\gamma-1}$ as $u\to0$ for
  some $\gamma>0$. Then,
  \begin{equation}
    \label{eq:9}
    \int_0^\infty h(u)e^{-ctu} du
    \asymp t^{-\gamma}\quad \text{as}\;\; t\to\infty,
  \end{equation}
  and, if $\tau>0$, 
  \begin{displaymath}
    H(v):=\int_0^v h(u)(v-u)^{\tau-1} du\asymp v^{\gamma+\tau-1}\quad \text{as}\;\;
    v\to0. 
  \end{displaymath}
  Furthermore, $H$ grows at most polynomially. 
\end{lemma}

As in Section~\ref{sec:centr-limit-theor}, denote by $\eta_t$ a
Poisson process with intensity measure $t\lambda$ for $t>0$, and
define $\hat{F}_t:=\hat{F}(\eta_t)$, see \eqref{eq:estim-Hoelder}.
Under suitable assumptions on $f$ we shall derive the variance
asymptotics and a central limit theorem for $\hat{F}_t$. For a
function $f\in L^1(\lambda)\cap L^2(\lambda)$, denote
\begin{displaymath}
  f_i(u):=\int \big|f(s,\phi(s)-u) \big|^i 
  \,ds,
  \quad u>0,\;i=2,3,4.
\end{displaymath}

\begin{theorem}
  \label{t-bounds-var} 
  For all $t>0$,
  \begin{multline}
    \label{eq:2}
    t\frac{\beta}{d+\beta}\int_0^\infty 
    f_2(v^{\beta/(d+\beta)})v^{-d/(d+\beta)} e^{-atv}\, dv
    \leq \BV \hat{F}_t\\ \le
    t\frac{\beta}{d+\beta}\int_0^\infty 
    f_2(v^{\beta/(d+\beta)})v^{-d/(d+\beta)} e^{-btv}\, dv,
  \end{multline}
  where 
  \begin{align*}
    a:=\frac{\beta\kappa_d}{d+\beta}(R-R')^{-d/\beta},\quad
    b:=\frac{\beta\kappa_d}{d+\beta}(2R)^{-d/\beta}.
  \end{align*}
\end{theorem}
\begin{proof}
  By Theorem~\ref{tvariance}, 
  \begin{align*}
    \BV \hat{F}_t&=t\iint \I \big\{u\le\phi(s) \big\}f(s,u)^2 \BE
    H_{s,u}(\eta_t) \,ds\,du\\
    &=t\iint \I\{u\ge 0\}f \big (s,\phi(s)-u \big)^2 \BE
    H_{s,\phi(s)-u}(\eta_t)\,ds\,du.
  \end{align*}
  Since Lemma~\ref{lemma:H-bound} applies to $\eta_t$
  with all bounds multiplied by $t$, we have 
  \begin{align*}
    \BV \hat{F}_t
    &\le t \int \int^\infty_0 f \big (s,\phi(s)-u \big)^2 
    \exp\big[-b t u^{(d+\beta)/\beta}\big] du\,ds\\
    &= t\int^\infty_0  f_2(u)
    \exp\big[-b t u^{(d+\beta)/\beta}\big]\,du.
  \end{align*}
  A change of variables yields the upper bound in \eqref{eq:2}.
  To derive the lower bound, we rely on the upper bound in 
  Lemma~\ref{lemma:H-bound} and proceed as above.
\end{proof}

Under the polynomial growth assumption,
the asymptotic behaviour of the variance as $t\to\infty$ is determined
by the behaviour of the function $f_2(u)$ as $u\to0$.

\begin{corollary}
  \label{cor:bound}
  Assume that $R'<R$, that $f_2$ grows at most polynomially, and
  $f_2(u)\asymp u^{\gamma-1}$ as $u\to0$ for some $\gamma>0$.  Then
  there exist $c_1,c_2>0$ such that
  \begin{displaymath}
    c_1t^{1-\gamma\beta/(d+\beta)}\leq \BV \hat{F}_t
    \le c_2t^{1-\gamma\beta/(d+\beta)},\quad t\geq 1.
  \end{displaymath}
\end{corollary}
\begin{proof}
  The asymptotic behaviour
  $\BV \hat{F}_t\asymp t^{1-\gamma\beta/(d+\beta)}$ as $t\to\infty$
  follows from \eqref{eq:2} and Lemma~\ref{lemma:Tauber} with $\gamma$
  replaced by $\beta\gamma/(d+\beta)$.
  Therefore, the above inequalities hold for
  $t\geq c$ with a sufficiently large $c$. They can be extended to
  $[1,c]$, since the bounds in Theorem~\ref{t-bounds-var} are continuous in
  $t>0$ and positive for all $t>0$.
\end{proof}

We continue with a quantitative central limit theorem. 
The proof of the following result relies on
Corollary~\ref{cor:normal-prime} and is given in the supplement.  As
before, we denote $\sigma_t^2:=\BV \hat{F}_t$.

\begin{theorem}
  \label{tcltfunction} 
  Assume that $R'<R$, the functions $f_i$, $i=1,\dots,4$, grow at most
  polynomially, and that
  \begin{equation}
    \label{eq:7}
    a_i\leq \liminf_{u\downarrow 0} u^{1-\gamma}f_i(u)\leq
    \limsup_{u\downarrow 0} u^{1-\gamma}f_i(u)\leq b_i, \quad i=2,3,4
  \end{equation}
  for some $\gamma\in(0,1]$ and with $0\leq a_i\leq b_i<\infty$ and
  $a_2>0$.  Then there exists a $c>0$ depending on $f$, $d$, $\beta$,
  and $R,R'$, such that
  \begin{displaymath}
    d_W\big(\sigma^{-1}_t(\hat{F}_t-tF),N\big)
    \le c t^{-1/2(1-\gamma\beta/(d+\beta))},\quad t\ge 1.
  \end{displaymath}
\end{theorem}

\begin{remark}
  Assume that $\int |f(s,\phi(s))|\,\nu(ds)>0$ and
  $f(s,\phi(s)-u)\to f(s,\phi(s))$ as $u\downarrow 0$ for $\nu$-a.e.\
  $s$.  If the function
  $s\mapsto \sup_{u\leq \phi(s)}|f(x,\phi(s)-u)|$ belongs to
  $L^2(\lambda)\cap L^4(\lambda)$ then condition \eqref{eq:7} is
  satisfied with $\gamma=1$. This holds in particular if $\nu$ is
  finite and $f$ is bounded.
\end{remark}

\begin{remark}
  Note that the rate in the Wasserstein distance is $\sigma_t^{-1}$,
  which indicates that it is likely optimal.  It is possible to bound
  the terms, assuming that \eqref{eq:7} holds with $\gamma_i$ instead
  of $\gamma$, so that the asymptotic behaviour of $f_i$ varies with
  $i$. However, then the obtained rates in the central limit theorem
  are no longer of the order $\sigma_t^{-1}$ and so are not
  necessarily optimal.
\end{remark}

\begin{remark}
  A central limit theorem for $\hat{F}_t$ for functions defined on the
  unit interval $[0,1]$ and in case of $\sigma_t^2$ growing as
  $\sqrt{t}$ was presented in \cite[Theorem~3.4]{ReissSelk17}. Our
  result provides the rate, holds in general dimension, applies for
  functions defined on the whole space,
  and does not rely on this
  particular rate for the variance. As described in
  Remark~\ref{rem:Kolm}, it is possible to show that the same rate
  holds for the Kolmogorov distance between
  $\sigma^{-1}_t(\hat{F}_t-tF)$ and $N$.
\end{remark}

\begin{example}
  Let $f(s,u)=\I\{u\geq 0\}$. If $\phi$ is integrable, then
  $f\in L^1(\lambda)\cap L^2(\lambda)$ and $F=\int \phi(s)ds$. In this
  case, $f_i(u)=\nu(\{s: \phi(s)\geq u\})$ for all $i\geq1$, where
  $\nu$ is the Lebesgue measure.  For instance, \eqref{eq:7} holds
  with $\gamma=1$ if $\nu(\{s\in\R:\varphi(s)>0\})<\infty$.
  Essentially, this is the setting of \cite{ReissSelk17}, where
  functions on $[0,1]$ have been
  considered. 
  If $\nu(\{s\in\R:\varphi(s)>0\})=\infty$, then $f_2(u)\to\infty$ as
  $u\downarrow0$.  Since
  \begin{displaymath}
    \int \phi(s)\,ds=\iint \I \big\{u\leq\phi(s) \big\}\, duds
    =\int \nu \big (\{s:\phi(s)\geq u\}\big)\, du<\infty,
  \end{displaymath}
  the function $f_2$ is integrable near zero and converges to
  infinity. Assume that $f_2(u)\asymp u^{\gamma-1}$, where necessarily
  we have $\gamma\in(0,1)$. Since $f_2$ is decreasing, the polynomial
  growth condition is evidently satisfied. Theorem~\ref{tcltfunction}
  yields a rate of convergence of the normalised estimation error to
  the normal distribution.
  Consider, for instance, the Lipschitz function
  $\phi(s)=1/(1+s^2)$. Then $f_2(u)\asymp u^{-1/2}$, so that
  $\gamma=1/2$.
\end{example}

\begin{example}
  Let $f(s,u)=pu_+^{p-1}$ with $p>1$, and assume that $\phi$ is a
  $p$-integrable and $(2p-1)$-integrable nonnegative function to
  ensure that $f\in L^1(\lambda)\cap L^2(\lambda)$. Then
  $F=\int \phi(s)^pds$. In this case,
  \begin{displaymath}
    f_i(u)=p^i \int \big (\phi(s)-u \big)_+^{i(p-1)}\, ds.
  \end{displaymath}
  Then \eqref{eq:7} holds with $\gamma=1$ if $\phi$ is $2(p-1)$- and
  $4(p-1)$-integrable. Noticing the previous integrability condition,
  we need to impose that $\phi$ is integrable of the orders
  $\min(p,2p-2)$ and $\max(4(p-1),2p-1)$. If $\phi$ is bounded, we
  need only integrability of the order $\min(p,2p-2)$. 
\end{example}

\subsection{Families of functions related to convex bodies}
\label{sec:famil-funct-relat}

In this subsection we specify the setting of
Subsection~\ref{sec:param-famil-funct} to functions describing convex
bodies.  This can be most conveniently done by using the \emph{support
  function}
\begin{displaymath}
  h_K(s):=\sup\{\langle s,x\rangle: x\in K\},\quad s\in\BS, 
\end{displaymath}
of a nonempty convex set $K\subset\R^d$, where $\BS$ is the unit
sphere in $\R^d$ and $\langle s,x\rangle$ stands for the scalar
product. We recall that the convex hull of the union of convex bodies
corresponds to taking pointwise maxima of their support functions and
the inclusion of convex bodies is equivalent to inequality between
their support functions.  Another way to describe a convex closed set
$K$ containing the origin is its \emph{radial function}
\begin{align*}
  \rho_K(s):=\sup\{u\ge 0: us\in  K\},\quad s\in\BS.
\end{align*}

\subsubsection{Approximation by convex hulls}

Let $L\subset \R^d$ be a convex body containing the origin.  Define a
family of functions on $\BS$ parametrised by $(s,u)\in \BS\times\R$ as
\begin{displaymath}
  g_{s,u}(r):=\max\big(h_L(r),u\langle s,r\rangle\big),\quad r\in\BS, u\ge0,
\end{displaymath}
and we set $g_{s,u}:=h_L$ if $u<0$. 
The choice $\varpi:=h_L$ provides a lower bound for $g_{s,u}$.
Let $\BX$ be the family of functions $g_{s,u}$.
Condition (F0) follows from the continuity of the parametrisation,
condition (F1) from the continuity of $g_{s,u}$. Furthermore, (F2)
holds since if $u\langle s,r\rangle=u'\langle s',r\rangle$ for $r$
from an open set, we have $(s,u)=(s',u')$.  Assumption (F3$^\prime$)
means that $\cX_0$ is the smallest ring containing
the set $\{h_L\}$ as well as the sets
$\{(s,u)\in\BS\times\R:us\notin L_\eps\}$, $\eps>0$, where $L_\eps$ is
the set of points whose distance from $L$ is at most $\eps$.

For $\mu\in\bN$, let $Z_{\mu}$ denote the convex hull of all points
$us$ such that $g_{s,u}\in\mu$, equivalently, the support function of
$Z_\mu$ is $\sup\mu$.  The generator \eqref{eq:24} defines
$\partial\mu$ as the family of $g_{s,u}\in\mu$ (accounting for
multiplicities) such that $us$ is a vertex of $Z_{\mu}$.
Note that the generator given by \eqref{eq:24} does not have the prime
property. The hull operator $[\mu]$ is the set of all $g_{s,u}$ such
that $us$ belongs to $\conv(L\cup Z_\mu)$ (the
convex hull of the union of $L$ and $Z_\mu$) excluding the vertices. 

Let $K\subset\R^d$ be another convex body with nonempty interior
and such that $L\subset K$,
equivalently, $h_L\leq h_K$.  Set $\phi:=h_K$.
Since $g_{s,u}$ is the support
function of the convex hull of $L$ and the point $us$, we
have $g_{s,u}\leq \phi$ if and only if $us\in K$.
Furthermore, $g_{s,u}=\varpi$ if and only if $us\in L$.
Hence, 
\begin{align*}
  \BH:=\big\{(s,u)\in\BS\times\R:\varpi \ne g_{s,u}\le \phi \big\}=
  \big\{(s,u)\in\BS\times\R:us\in K\setminus L \big\}
\end{align*}
and
\begin{displaymath}
  \phi^-(s)=\sup\{u\in\R: g_{s,u}\leq \phi\}=\sup\{u\ge 0: us\in K\}=\rho_K(s),\quad s\in\BS,
\end{displaymath}
is the radial function of $K$. Similarly, $\varpi_*=\rho_L$ is the
radial function of $L$.

Define a measure $\bar\lambda$ on $\BS\times \R$ by letting
$\bar\lambda(d(s,u)):=\I\{u\ge 0\}\nu(ds) u^{d-1}du$, where $\nu$ is the
$(d-1)$-dimensional Hausdorff measure on $\BS$.  Denote by $\lambda$
the restriction of $\bar\lambda$ onto the set $\BH$, noticing that
each point of this set corresponds to a function from $\BX$. Then 
  $\lambda$ is finite and condition (F4) holds.  Let
$\eta$ be a Poisson process with intensity measure $\lambda$.  Then
$\{us: (s,u)\in\eta\}$ is a homogeneous unit intensity Poisson
process on $K\setminus L$.

Let $f(s,u):=w(s)\beta u^{\beta-d}$, $(s,u)\in\BS\times\R$, with
$\beta>0$ and a function $w\in L^1(\nu)$. Then we have
$f\in L^1(\lambda)$ and the functional \eqref{eq:20} becomes
\begin{equation}
  \label{eq:21}
  F=\int_{\BS}\int_{\rho_L(s)}^{\rho_K(s)}f(s,u) u^{d-1}\, \nu(ds) du
  =\int_{\BS} w(s) \Big(\rho_K(s)^\beta-\rho_L(s)^\beta\Big)\,\nu(ds).
\end{equation}
Since $(\sup\eta)^-$ is the radial function of $\conv(L\cup Z_\eta)$,
the Poisson hull estimator of $F$ is
\begin{displaymath}
  \hat{F}=\int_{\BS} w(s)
  \Big(\rho_{\conv(L\cup Z_\eta)}(s)^\beta-\rho_L(s)^\beta\Big)\,\nu(ds)
  +\int w(s)  \beta u^{\beta-d}\; \partial\eta(d(s,u)). 
\end{displaymath}
If $\beta=d$, $w(s)=1/d$ for all $s$, and $L=\{0\}$, then $F$ equals
the volume $V_d(K)$ of $K$ and $\hat{F}$ is the sum of the volume of
the convex hull of points from $\eta+\delta_0$ and the cardinality of
the number of vertices in this convex hull which are distinct from the
origin. In comparison with the oracle estimator for the volume of $K$
suggested in \cite{BaldinReiss16} (see also
Example~\ref{ex:chull-oracle}), the first term of $\hat{F}$ may be
larger, while the second term may be smaller, since the origin is
excluded from the generator.  Assume now that $\beta>d/2$ and
$w\in L^2(\nu)$, so that $f\in L^2(\lambda)$.  The variance of
$\hat{F}$ is then given by $\BE V_d(K\setminus \conv(L\cup Z_\mu))$.
It is smaller than the variance of the estimator in
\cite{BaldinReiss16}, since we utilise extra information that $K$
contains the set $L$.

If $L$ contains the origin in its interior, then the function
$f(s,u)=u^{\beta-d}$ belongs to $L^2(\lambda)$ for all
$\beta\in\R$. For instance, if $\beta=-1$, we obtain an estimator for
\begin{displaymath}
  F=\int_{\BS} \big(\,\rho_L(s)^{-1}-\rho_K(s)^{-1}\big)\,\nu(ds),
\end{displaymath}
which is proportional to the difference between the mean widths of
polar bodies to $L$ and $K$, see \cite[Page~616]{sch:weil08}.

\subsubsection{Approximation by intersections of half-spaces}

Consider now a dual approximation of a convex body from the outside.
In this case we are in the setting of infimum of functions instead of
supremum, see Remark~\ref{rem:minimum}.  Fix a convex closed set $L$
which contains the origin in its interior and can be unbounded, and let
$\varpi:=\rho_L$ be the upper boundary function.  Let
\begin{displaymath}
  g_{s,u}(r):=\min \big(\rho_L(r), u\langle s,r\rangle_+^{-1}\big),
  \quad s\in\BS,\; u\ge0,
\end{displaymath}
so that $g_{s,u}$ is the radial function of the intersection of $L$
and the half-space $H^-(s,u):=\{x\in\R^d: \langle x,s\rangle\leq u\}$.
Due to the change of order, the functions $g_{s,u}$ are allowed to
take infinite values. If $u<0$, let $g_{s,u}:=0$.  We leave to the
reader to check the adapted conditions (F1) and (F2) and to figure out
the meaning of the adjusted assumption (F3$^\prime$).  For instance,
if $L$ is compact, then a measure $\mu$ on $\BX$ is locally finite if
$\mu(\{(s,u): u\ge 0, H^-(s,u)\cap L\ne\emptyset\})<\infty$.  Note
that $\varpi_*(s)=\sup\{u\in\R:\varpi\ne g_{s,u}\}$ is the smallest
value of $u$ such that $L\subset H^-(s,u)$, so that
$\varpi_*(s)=h_L(s)$, which is the support function of $L$.

Consider the measure
$\bar\lambda(d(s,u)):=\I\{u\ge 0\}\nu(ds) du$ on
$\BS\times\R$, where $\nu$ is the $(d-1)$-dimensional
Hausdorff measure $\nu$ on $\BS$. This measure defines a stationary
(and isotropic) Poisson process on the affine Grassmannian $A(d,d-1)$, see
Example~\ref{ex:polytope-stopping}.

The aim is to recover information about an unknown convex body $K$
which contains the origin in its interior and such that $K\subset L$. Let
$\phi:=\rho_K$.  Then $\phi\leq g_{s,u}$ if and only if $K\subset
H^-(s,u)$, that is, $u\geq h_K(s)$. Thus, $\phi^+=h_K$.
Let $\lambda$ be the restriction of $\bar\lambda$ to
\begin{align*}
  \BH:=\big\{(s,u)\in\BS\times[0,\infty): h_K(s)\leq u\leq h_L(s)\big\},
\end{align*}
so that $\lambda$ determines a Poisson process on $(d-1)$-dimensional
affine hyperplanes which do not intersect $K$ and intersect $L$,
equivalently, on the family of all half-spaces which contain $K$ and
do not contain $L$. The intersection of all such half-spaces is a random set
$P_\eta$ called the Poisson polytope, which almost surely contains
$K$, see \cite{MR4134241}.
Since $\partial\eta$ is a.s.\ finite (see
Example~\ref{ex:polytope-stopping}), condition (F4) is satisfied.
The radial function of $P_\eta$ equals $\inf\eta$ and
$(\inf\eta)^+=h_{P_\eta}$.

Let $f(s,u):=w(s) \beta u^{\beta-1}$, $(s,u)\in\BS\times\R$,
with $\beta<0$ and a function $w\in L^1(\nu)$. Then an analogue of the
functional \eqref{eq:20} becomes
\begin{equation}
  \label{eq:21a}
  F=\int_{\BS}\int_{h_K(s)}^{h_L(s)}f(s,u) \; \nu(ds) du
  =\int_{\BS} w(s) \big(h_K(s)^\beta-h_L(s)^\beta\big)\,\nu(ds).
\end{equation}
If $\beta=-d$, $w(s)\equiv1$, and $L=\R^d$, then $F$ is the integral
of $h_K^{-d}$, which is proportional to the volume of the polar body
to $K$.  If $L$ is bounded, it is possible to consider any
$\beta\neq0$. For instance, if $\beta=1$ and $w(s)\equiv1$, then
\begin{displaymath}
  F=\int_{\BS} 
  \big(h_L(s)-h_K(s)\big)\,\nu(ds),
\end{displaymath}
which is the difference between the mean widths of $L$ and $K$; if
also $d=2$, then $F$ is the difference between the perimeters of $L$
and $K$.  The Poisson hull estimator of $F$ becomes
\begin{displaymath}
  \hat{F}=\int_{\BS} 
  \big(h_L(s)-h_{P_\eta}(s)\big)\, ds
  -\card(\partial\eta).
\end{displaymath}

\bigskip

\section*{Acknowledgments}

The authors are grateful to two anonymous referees for several
  corrections and encouraging suggestions to the first version of
  this work and to Andrei Ilienko for several critical remarks.
  IM is grateful to Mathematics Department of the Karlsruhe Institute
  of Technology for hospitality.  

The second author was supported by Swiss National Science Foundation grant
200021\_175584 and the Alexander von Humboldt Foundation.


\newpage

\section*{Supplementary material for the paper: Poisson hulls}

\addtocounter{section}{1}
\setcounter{equation}{0}
\setcounter{proposition}{0}
\renewcommand{\theequation}{\arabic{equation}}
\renewcommand{\theproposition}{\arabic{proposition}}



  





\subsection*{Proof of Proposition~3.3}

We need to show that
\begin{align}\label{e3.5}
  \BE h(\partial\eta,\eta_{[\eta]})
  =\BE\int h(\partial\eta,\mu_{[\eta]})\,\Pi_\lambda(d\mu)
\end{align}
for all bounded and measurable $h\colon\bN^2\to\R$.  For
$B\in\mathcal{X}$, let $\mathcal{N}_B\subset\mathcal{N}$ be the
$\sigma$-field generated by the mapping $\mu\mapsto\mu_B$.  The
monotone class theorem easily implies that
$\cup_{m\in\N}\mathcal{N}_{B_m}$ contains an intersection stable
family generating $\mathcal{N}$. Therefore, we can assume that there
exists an $m\in\N$ such that $h(\mu,\psi)=h(\mu_B,\psi_B)$ for all
$\mu,\psi\in\bN$, where $B=B_m$. Let $\eta_n:=\eta_{B_n}$ be the
restriction of $\eta$ to $B_n$. For each $n\in\N$, define the event
\begin{align*}
  D_n:=\big\{(\partial \eta_n)_B=(\partial \eta)_B,\; 
  \eta_{[\eta_n]\cap B}=\eta_{[\eta]\cap B}\big\}. 
\end{align*}
Then 
\begin{equation}
  \label{eq:19a}
  \BE h(\partial\eta,\eta_{[\eta]})
  =\BE\big[\I_{\Omega\setminus D_n}
  \big(h(\partial\eta,\eta_{[\eta]})
  -h(\partial\eta_n,(\eta_n)_{[\eta_n]})\big)\big]
  +\BE\big[h(\partial\eta_n,(\eta_n)_{[\eta_n]})\big].
\end{equation}
For each $n\in\N$ and each $\mu\in\bN$, define the event
\begin{displaymath}
  E_n(\mu):=\{\mu_{[\eta]\cap B}= \mu_{[\eta_n]\cap B}\}.
\end{displaymath}
Then 
\begin{align}
  \label{eq:19b} \notag
  \BE\int h(\partial\eta,\mu_{[\eta]})\,\Pi_\lambda(d\mu)
  &=\BE\bigg[\int \I_{\Omega\setminus (D_n\cup E_n(\mu))}
    \big(h(\partial\eta,\mu_{[\eta]}) 
    -h(\partial\eta_n,\mu_{[\eta_n]})\big)\,\Pi_\lambda(d\mu)\bigg]\\
    &\qquad\qquad +\BE\int h(\partial\eta_n,\mu_{[\eta_n]})\,\Pi_\lambda(d\mu).
\end{align}
The first terms on the right-hand sides of \eqref{eq:19a} and
\eqref{eq:19b} converge to zero as $n\to\infty$ by conditions (3.3)
and (3.4) imposed in Proposition 3.3 from \cite{las:mol:23}.  The
second terms coincide by Theorem~3.2 from \cite{las:mol:23}, since
\eqref{e3.5} holds with $\eta_n$ in place of $\eta$.

\subsection*{Proof of Theorem~6.1} 

By the conditional variance formula,
\begin{align}\label{e5.5}
  \BV \hat{F}^{(k)}= \BV\int f\,d\eta^{(k)} 
  -\BE \BV\left[ \int f\,d\eta^{(k)}\Bigm\vert \partial\eta\right].
\end{align}
Using conditional covariances we can write
\begin{align}
  \label{e5.92}\notag
  X:=\BV&\left[ \int f\,d\eta^{(k)}\Bigm\vert \partial\eta\right]
          =\BV\left[\sum^k_{i=0}\binom{k}{i}\iint f(\bx,\bv)\,(\partial\eta)^{(k-i)}(d\bv)\,
          (\eta_{[\eta]})^{(i)}(d\bx)\Bigm\vert \partial\eta\right]\\
        &=\sum^k_{i,j=1}\binom{k}{i}\binom{k}{j}
          \BC\left[\int f_i(\bx)\,(\eta_{[\eta]})^{(i)}(d\bx),
          \int f_j(\bx)\,(\eta_{[\eta]})^{(j)}(d\bx)\Bigm\vert \partial\eta\right],
\end{align}
where
\begin{align*}
  f_i(\bx):=\int f(\bx,\bv)\,(\partial\eta)^{(k-i)}(d\bv),\quad \bx\in\BX^i,\,
  i\in\{1,\ldots,k\}.
\end{align*}

To proceed, we need a formula for covariances of Poisson U-statistics.  Let
$i,j\in\N$, and let $g\colon\BX^i\to\R$ and $h\colon\BX^j\to\R$ be
measurable and symmetric.  Under suitable integrability assumptions,
it follows from Proposition~12.11 and Corollary~12.8 in
\cite{LastPenrose17} that
\begin{align}
  \label{e5.78}\notag
  \BC&\left[\int g\,d\eta^{(i)},\int h\,d\eta^{(j)}\right]\\
     &=\sum^{i\wedge j}_{n=1}\binom{i}{n}\binom{j}{n}n!
       \iiint g(\bx,\by)h(\bx,\bz)\,\lambda^{i-n}(d\by)\,
       \lambda^{j-n}(d\bz)\,\lambda^{n}(d\bx).
\end{align}
This formula generalises \cite[Eq.~(6.4)]{las:mol:23}.
Using \cite[Theorem~3.2]{las:mol:23} 
and \eqref{e5.78} in \eqref{e5.92}, we obtain
that
\begin{align*}
  X=&\sum^k_{i,j=1}\binom{k}{i}\binom{k}{j}
      \sum^{i\wedge j}_{n=1}\binom{i}{n}\binom{j}{n}n!
      \idotsint f(\bx,\by,\bv)f(\bx,\bz,\bw)\\ 
    &\times\bH_{\bx}(\eta)\bH_{\by}(\eta)\bH_{\bz}(\eta)
      \,\lambda^{i-n}(d\by)\,\lambda^{j-n}(d\bz)\,\lambda^{n}(d\bx)
      \,(\partial\eta)^{(k-i)}(d\bv)\,(\partial\eta)^{(k-j)}(d\bw),
\end{align*}
where
\begin{align*}
  \bH_{\bx}(\eta):=\prod^n_{l=1}\bH_{x_l}(\eta),\quad
  \bx=(x_1,\ldots,x_n)\in\BX^n.
\end{align*}

Next, we need a property of factorial measures. For $r,s\in\N$ and
$l\in\{0,\ldots,r\wedge s\}$, let $A_{r,s,l}$ be the set of all
$(\bv,\bw)\in\BX^r\times\BX^s$ such that the total variation
distance between $\delta_{\bv}$ and $\delta_{\bw}$ equals
$r+s-2l$. The relation $(\bv,\bw)\in A_{r,s,l}$ means that $\bv$ and
$\bw$, when interpreted as multisets, coincide in exactly $l$
points.  Let $g\colon\BX^r\to\R$ and $h\colon\BX^s\to\R$ be
measurable and symmetric.  Then
\begin{multline}
  \label{e5.8}
  \iint \I\{(\bv,\bw)\in A_{r,s,l}\}g(\bv)h(\bw)\,\mu^{(r)}(d\bv)\,\mu^{(s)}(d\bw)\\
  =l!\binom{r}{l}\binom{s}{l}\int \I\{\bu\in \BX^l\}g(\bu,\bv)h(\bu,\bw)
  \,\mu^{(r+s-l)}(d(\bu,\bv,\bw)),
  \quad \mu\in\bN.
\end{multline}
If $\mu$ is a finite sum of Dirac measures, this is a purely
combinatorial fact, whose proof is left to the reader. The general
case follows from \cite[Lemma~A.15]{LastPenrose17}.  Since
$\BX^r\times\BX^s$ is the disjoint union of the sets $A_{r,s,l}$, we
obtain from \eqref{e5.92} and \eqref{e5.8} that $X$ equals
\begin{multline*}
  \sum^k_{n=1}\sum^k_{i,j=1}\sum^{k}_{l=0}
    \binom{k}{i}\binom{k}{j}\binom{i}{n}\binom{j}{n}n!
    l!\binom{k-i}{l}\binom{k-j}{l}
    \idotsint f(\bx,\by,\bu,\bv)f(\bx,\bz,\bu,\bw)\\
  \times\I\{\bu\in\BX^l\}\bH_{\bx}(\eta)\bH_{\by}(\eta)\bH_{\bz}(\eta)
    \,\lambda^{i-n}(d\by)\,\lambda^{j-n}(d\bz)\,\lambda^{n}(d\bx)
    \,(\partial\eta)^{(2k-i-j-l)}(d(\bu,\bv,\bw)).
\end{multline*}
In the above sum we have $l\le k-i\le k-n$.  Substituting
$m=l+n(\le k)$ in the inner sum and swapping the order of summation
yield that $X$ equals
\begin{multline*}
  \sum^k_{m=1}\sum^m_{n=1}\sum^k_{i,j=1}\binom{k}{m}^2m!
    \binom{m}{n}\binom{k-m}{i-m}\binom{k-m}{j-m}
    \idotsint f(\bx,\bu,\by,\bv)f(\bx,\bu,\bz,\bw)\\
  \times\I\{\bu\in\BX^{m-n}\}\bH_{\bx}(\eta)\bH_{\by}(\eta)\bH_{\bz}(\eta)
    \,\lambda^{i-n}(d\by)\,\lambda^{j-n}(d\bz)\,\lambda^{n}(d\bx)
    \,(\partial\eta)^{(2k-i-j-m+n)}(d(\bu,\bv,\bw)),
\end{multline*}
where we have used that
\begin{align*}
  \binom{k}{i}\binom{k}{j}\binom{i}{n}\binom{j}{n}n!
  (m-n)!\binom{k-i}{m-n}\binom{k-j}{m-n}
  =\binom{k}{m}^2m!\binom{m}{n}\binom{k-m}{i-n}\binom{k-m}{j-n}.
\end{align*}
Taking expectations and using the multivariate Mecke equation (and
changing the summation indices $(i,j)$ to $(i-n,j-n)$) yield that
\begin{align*}
  \BE X=
  &\sum^k_{m=1}\sum^m_{n=1}\sum^{k-m}_{i,j=0}\binom{k}{m}^2m!
    \binom{m}{n}\binom{k-m}{i}\binom{k-m}{j}
    \idotsint f(\bx,\bu,\by,\bv)f(\bx,\bu,\bz,\bw)\\
  &\qquad\qquad\qquad\times\BE\Big[\bH_{\bx}(\eta_{\bu,\bv,\bw})H_{\bu}(\eta_{\bu,\bv,\bw})
    \bH_{\by}(\eta_{\bu,\bv,\bw})
    H_{\bv}(\eta_{\bu,\bv,\bw})\bH_{\bz}(\eta_{\bu,\bv,\bw})H_{\bw}(\eta_{\bu,\bv,\bw})\Big]\\
  &\qquad\qquad\qquad \times \lambda^i(d\by)\,\lambda^{k-m-i}(d\bv)\,
    \lambda^j(d\bz)\,\lambda^{k-m-j}(d\bw)\,
    \lambda^{n}(d\bx)\,\lambda^{m-n}(d\bu),
\end{align*}
where
$\eta_{\bu,\bv,\bw}:=\eta+\delta_{\bu}+\delta_{\bv}+\delta_{\bw}$.  We
assert that
\begin{align}
  \label{eLemma}
  \bH_{\bx}(\eta_{\bu,\bv,\bw})
  &H_{\bu}(\eta_{\bu,\bv,\bw})\bH_{\by}(\eta_{\bu,\bv,\bw})
    H_{\bv}(\eta_{\bu,\bv,\bw})\bH_{\bz}(\eta_{\bu,\bv,\bw})H_{\bw}(\eta_{\bu,\bv,\bw})\notag\\
  &=\bH_{\bx}(\eta_{\bx,\bu,\by,\bv,\bz,\bw})H_{\bu}(\eta_{\bx,\bu,\by,\bv,\bz,\bw})
    \bH_{\by}(\eta_{\bx,\bu,\by,\bv,\bz,\bw})
    H_{\bv}(\eta_{\bx,\bu,\by,\bv,\bz,\bw})\notag \\
  &\qquad\qquad\qquad\qquad\qquad\qquad\qquad
    \times \bH_{\bz}(\eta_{\bx,\bu,\by,\bv,\bz,\bw})
    H_{\bw}(\eta_{\bx,\bu,\by,\bv,\bz,\bw}).
\end{align}
To see this we apply Lemma~\ref{l974} (to be proved below) with
$\mu:=\eta_{\bu,\bv,\bw}$.  Then
\begin{displaymath}
  \bH_{\bx}(\mu)\bH_{\by}(\mu)\bH_{\bz}(\mu)=
  \bH_{\bx}(\mu+\delta_{\bx}+\delta_{\by}+\delta_{\by})
  \bH_{\by}(\mu+\delta_{\bx}+\delta_{\by}+\delta_{\bz})
  \bH_{\bz}(\mu+\delta_{\bx}+\delta_{\by}+\delta_{\bz}).
\end{displaymath}
Moreover, by Lemma~\ref{l974} we also have
$\partial\mu=\partial(\mu+\delta_{\bx}+\delta_{\by}+\delta_{\bz})$,
so that by (H4)
\begin{align*}
  \partial(\mu+\delta_{u_i})
  =\partial(\mu+\delta_{\bx}+\delta_{\by}+\delta_{\bz}+\delta_{u_i})
\end{align*}
for each component $u_i$ of $\bu$. Therefore, we obtain from
\cite[Lemma~2.4]{las:mol:23} 
that
$H_{\bu}(\mu)=H_{\bu}(\mu+\delta_{\bx}+\delta_{\by}+\delta_{\bz})$ and
\eqref{eLemma} follows.

We now rename the variables as $\bx:=(\bx,\bu)$, $\by:=(\by,\bv)$, and
$\bz:=(\bz,\bw)$. The symmetry property of product measures yields
that
\begin{align}\label{e5.43}
  \BE X= \sum^k_{m=1}\binom{k}{m}^2m!\iiint
  f(\bx,\by)f(\bx,\bz) \BE S(\bx,\by,\bz)\,
  \lambda^{k-m}(d\by)\,\lambda^{k-m}(d\bz)\,\lambda^m(d\bx),
\end{align}
where 
\begin{align*}
  S(\bx,\by,\bz):=\sum^m_{n=1}
  \sum_{\substack{I\subset[m]\\|I|=n}}
  \prod_{l\in I}\bH_{x_l}\prod_{l\notin I}H_{x_l}
  \sum^{k-m}_{i=0}\sum_{\substack{J\subset[k-m]\\|J|=i}}
  \prod_{l\in J}\bH_{y_l}\prod_{l\notin J}H_{y_l}
  \sum^{k-m}_{j=0}\sum_{\substack{K\subset[k-m]\\|K|=j}}
  \prod_{l\in K}\bH_{z_l}\prod_{l\notin K}H_{z_l},
\end{align*}
and the argument $\eta_{\bx,\by,\bz}$ is dropped for notational
convenience.  For any numbers $a_1,\ldots,a_n\in[0,1]$ we have
(letting $\bar{a}_i:=1-a_i$)
\begin{align*}
  1=\prod^n_{i=1}(a_i+\bar{a}_i)
  =\sum_{J\subset[n]}\prod_{l\in J}a_l\prod_{l\notin J}\bar{a}_l. 
\end{align*}
Hence, the above inner sum over $(j,K)$ is one, and so
is the sum over $(i,J)$.  The remaining sum gives
\begin{align*}
  1-\prod^m_{l=1}H_{x_l}(\eta_{\bx,\by,\bz}).
\end{align*}
We can now insert this into \eqref{e5.43} and then in turn into
\eqref{e5.5}.  Taking into account the variance formula
\cite[Eq.~(6.4)]{las:mol:23},
we obtain \cite[Eq.~(6.5)]{las:mol:23}.

\begin{lemma}
  \label{l974}
  Let $\mu\in\bN$ and $\bx=(x_1,\ldots,x_n)\in\BX^n$. Then
  $\bH_{\bx}(\mu+\delta_{\bx})=1$ if and only if $\bH_{\bx}(\mu)=1$.
  In  this case $\partial(\mu+\delta_{\bx})=\partial\mu$.
\end{lemma}
\begin{proof}
  Assume first that $\prod^n_{i=1}\bH_{x_i}(\mu)=1$. It follows by
  (H4) and induction that
  \begin{align*}
    \partial\Big(\mu+\sum_{i\in I}\delta_{x_i}\Big)=\partial\mu
  \end{align*}
  for all non-empty $I\subset\{1,\ldots,n\}$. In particular,
  $\partial(\mu+\delta_{\bx})=\partial\mu$ and
  \begin{align}\label{e:345}
    \partial(\mu+\delta_{\bx})=\partial(\mu +\delta_{\bx}-\delta_{x_i}),\quad i=1,\ldots,n.
  \end{align}
  Equation \eqref{e:345} and \cite[Lemma~2.4]{las:mol:23} 
  show that $\bH_{x_i}(\mu+\delta_{\bx})=1$ for each
  $i\in\{1,\ldots,n\}$.

  Assume, conversely that \eqref{e:345} holds.  We assert that then
  \begin{align*}
    \partial\Big(\mu+\sum_{i\in I}\delta_{x_i}\Big)
    =\partial\Big(\mu+\sum_{i\in I}\delta_{x_i}-\delta_{x_j}\Big),
    \quad i=1,\ldots,n.
  \end{align*}
  For each non-empty $I\subset\{1,\ldots,n\}$ and each $j\in
  I$. This follows by backwards induction on the cardinality of $I$
  using the implication (i)$\rightarrow$(iii) from
  \cite[Lemma~2.8]{las:mol:23}.
  For $|I|=1$ we obtain
  $\prod^n_{i=1}\bH_{x_i}(\mu)=1$.
\end{proof}

\subsection*{Proof of Theorem~7.1 and Corollary~7.3}

We apply the results and use
the notation of \cite[Corollary~2.2]{las:mol:sch20} which provides the normal
approximation of the Kabanov--Skorohod integral of $G$ as
\begin{displaymath}
  d_W(\deltaKS(G),N)\le T_1+T_3+T_4+T_5,
\end{displaymath}
where the summands on the right-hand side are defined in
\cite{las:mol:sch20} and will be specified later on in the course of
their calculations. The normalised deviation $(\hat{F}_t-tF)/\sigma_t$
is the Kabanov--Skorohod integral of the functional
$G(x,\eta_t):=f(x)H_x(\eta)/\sigma_t$ with respect to the Poisson
process of intensity $t\lambda$. This results in normalising the
terms from \cite[Corollary~2.2]{las:mol:sch20} by appropriate powers
of $t$ and $\sigma_t$.  Our integrability conditions
correspond to those imposed in
\cite[Equations~(2.2)--(2.5)]{las:mol:sch20}.

We write shortly $H_x$ for $H_x(\eta_t)$, $H_x(y)$ for
$H_x(\eta_t+\delta_y)$, and $H_x(y,z)$ instead of
$H_x(\eta_t+\delta_y+\delta_z)$, possibly with other subscripts and
arguments and similarly for $\bH$. Since $H_x$ takes values $0$ or
$1$, \cite[Eq.~(2.10)]{las:mol:23}
yields that
\begin{align*}
  T_1^2&:=\frac{t^3}{\sigma_t^4}
         \int \BE \Big(\int f(x)^2(D_zH_x^2)
         \lambda(dx)\Big)^2 \lambda(dz)\\
       &=\frac{t^3}{\sigma_t^4}
         \int \BE \Big(\int f(x)^2 H_x \bH_x(z)
         \lambda(dx)\Big)^2 \lambda(dz)\\
    &=\frac{t^3}{\sigma_t^4}
      \int f(x)^2f(y)^2 \BE \big[H_x\bH_x(z)
      H_y\bH_y(z)\big]\;\lambda^3(d(x,y,z)).
\end{align*}
Furthermore,
\begin{displaymath}
  T_3:=\frac{t}{\sigma_t^3}\int |f(x)|^3\BE H_x\;\lambda(dx).
\end{displaymath}
The next term is given by 
\begin{align*}
  T_4&:=\frac{t^2}{\sigma_t^3}
  \BE\int \Big(2f(x)^2|f(y)\,|H_x^2|D_x H_y|+|f(x)|f(y)^2\,
  \big|H_xD_xH_y\big|\Big(2|H_y|+|D_xH_y|\Big)\Big)\;\lambda^2(d(x,y))\\
  &\leq \frac{t^2}{\sigma_t^3}
  \int \Big(2f(x)^2|f(y)|+3|f(x)|f(y)^2)\Big)
  \BE \big[H_xH_y\bH_y(x)\big]\;\lambda^2(d(x,y)).
\end{align*}
It remains to notice that
\begin{equation}
  \label{eq:23}
  H_xH_y\bH_y(x)=H_y\bH_y(x),
\end{equation}
since
\begin{align*}
  H_y\bH_y(x)-H_xH_y\bH_y(x)
  =\bH_xH_y\bH_y(x)
  =\bH_xH_y\bH_y(x)\bH_x(y)
  =\bH_xH_y\bH_x\bH_y=0,
\end{align*}
where we used (H4) and Equation (2.13) from \cite{las:mol:23}.

The final term can be written as
\begin{displaymath}
  T_5:=2\frac{t^3}{\sigma_t^3}
  \int \big|f(x)f(y)f(z)\big|\BE\big[A_t(x,y,z)B_t(x,y,z)\big]\;\lambda^3(d(x,y,z)),
\end{displaymath}
where 
\begin{align*}
  A_t(x,y,z):\!&=|D_yH_z|+|D^2_{x,y}H_z|,\\
  B_t(x,y,z):\!&=\big|D_z(H_xD_xH_y)\big|+2|H_xD_xH_y|\\
  &=\Big|D_zH_xD_xH_y+H_x(z)D^2_{x,z}H_y\Big|
              +2|H_xD_xH_y|.
\end{align*}
The last equality follows from the product rule for the difference
operator, see \cite{Last16}.
By \cite[Eq.~(2.11)]{las:mol:23}
with $m=2$ and splitting the cases of $\bH_z(y)$
being zero or one, we obtain
\begin{align*}
  D^2_{x,y}H_z&=H_z\big(\bH_z(x)+\bH_z(y)-\bH_z(x,y)\big)\\
  &=\bH_z(y) H_z\big(\bH_z(x)+\bH_z(y)-\bH_z(x,y)\big)
  +H_z(y) H_z\big(\bH_z(x)+\bH_z(y)-\bH_z(x,y)\big)\\
  &=\bH_z(y) H_z \bH_z(x)
  +H_z(y) H_z\big(\bH_z(x)-\bH_z(x,y)\big).
\end{align*}
By (H4), $\bH_z(x)-\bH_z(x,y)=-H_z(x) \bH_z(x,y)$, so that
\begin{equation}
  \label{eq:d2}
  D^2_{x,y}H_z=H_z\Big(\bH_z(x)\bH_z(y)
  -H_z(x)\bH_z(x,y)H_z(y)\Big).
\end{equation}
Hence
\begin{align*}
  A_t(x,y,z)
  &=H_z\Big(\bH_z(y)
    +\big|\bH_z(x)\bH_z(y)
    -H_z(x)H_z(y)
    \bH_z(x,y)\big|\Big)\\
  &=:H_z\big(A'+|A''-A'''|\big).
\end{align*}
Using \eqref{eq:d2} (with suitably amended subscripts and arguments of
$\bH$) and the fact that $H_x(z)=1$ implies $H_x=1$, we obtain
\begin{align*}
  B_t(x,y,z)
  &=\Big|H_x\bH_x(z)H_y\bH_y(x)
    +H_xH_x(z)D^2_{x,z}H_y\Big|
    +2H_xH_y\bH_y(x)\\
  &=H_xH_y\Big(\,\Big|\bH_x(z)\bH_y(x)
    +H_x(z)\bH_y(z)\bH_y(x)\\
  &\qquad\qquad\qquad\qquad -H_x(z)H_y(x)H_y(z)
    \bH_y(x,z)\Big|
    +2\bH_y(x)\Big).
\end{align*}
By considering each summand in the expression of $B_t(x,y,z)$
separately, it is easy to see that
\begin{displaymath}
  H_zA'''B_t(x,y,z)=0.
\end{displaymath}
For this, we use \cite[Eq.~(2.13)]{las:mol:23}
several times together with
\begin{align*}
  \bH_z(y)\bH_y(x,z)
  =\bH_z(y)\bH_z(y,x)\bH_y(x,z)
  =\bH_z(y)\bH_z(x)\bH_y(x)
\end{align*}
applied with various subscripts and arguments of $\bH$. Thus,
the term $A_t(x,y,z)$ can be replaced with $H_z\bH_z(y)\big(1+\bH_z(x)\big)$.
Furthermore,
\begin{align*}
  H_xH_yH_z\bH_z(y)
  & H_x(z)\bH_y(z)\bH_y(x)
  =H_xH_yH_z\bH_z
  H_x(z)\bH_y\bH_y(x) = 0,\\
  H_xH_yH_z\bH_z(y)
  & H_x(z)H_y(x)H_y(z)
    \bH_y(x,z)
    =H_xH_yH_z\bH_z(y)\bH_z(x,y)H_x(z)H_y(x)H_y(z)
    \bH_y(x,z)\\
  &=H_xH_yH_z\bH_z(y)\bH_z(x)H_x(z)H_y(x)H_y(z)
    \bH_y(x)=0.
\end{align*}
Hence,
\begin{multline*}
  A_t(x,y,z)B_t(x,y,z)
  =H_xH_yH_z\bH_z(y)\big(1+\bH_z(x)\big)
    \big(\bH_x(z)+2\big)\bH_y(x)\\
  =H_xH_yH_z\bH_z(y)\bH_y(x)
    \Big(2+2\bH_z(x)+\bH_x(z)
    +\bH_z(x)\bH_x(z)\Big).
\end{multline*}
The factor in parentheses is at most $4$, since
\begin{displaymath}
  H_xH_yH_z\bH_z(y)\bH_y(x)
  \bH_z(x)\bH_x(z)=H_xH_yH_z\bH_z(y)\bH_y(x)
  \bH_z\bH_x= 0.
\end{displaymath}
Thus,
\begin{displaymath}
  A_t(x,y,z)B_t(x,y,z)
  \leq 4H_xH_yH_z\bH_z(y)\bH_y(x)\bH_x(z).
\end{displaymath}
Finally, iterating the argument from \eqref{eq:23} twice, we have
\begin{displaymath}
  H_xH_yH_z\bH_z(y)\bH_y(x)
  =H_z\bH_z(y)\bH_y(x).
\end{displaymath}

\bigskip

We now turn to the proof of \cite[Corollary~7.3]{las:mol:23}.
In the prime setting, the simpler
expressions of difference operators make it possible to formulate
\cite[Eq.~(7.1)]{las:mol:23}
as
\begin{displaymath}
  \int f(y)^2 \BE H_y(\eta)\bH_y(\delta_x)\,\lambda(d(x,y))
  =\int f(y)^2 \BE H_y(\eta)h_0(y)\,\lambda(dy)<\infty
\end{displaymath}
and \cite[Eq.~(7.2)]{las:mol:23}
as
\begin{equation}
  \label{eq:int-3-bis}
  \int f(y)^2 \BE H_y(\eta)\bH_y(\delta_x)\bH_y(\delta_z)\,\lambda(d(x,y,z))
  =\int f(y)^2 \BE H_y(\eta)h_0(y)^2\,\lambda(dy)<\infty.
\end{equation}
The latter condition is imposed in \cite[Eq.~(7.6)]{las:mol:23},
while the
first one follows from it, given that $f$
is square integrable and since $h_0(y)\leq\min(1,h_0(y)^2)$.
Finally,  \cite[Eq.~(7.3)]{las:mol:23}
becomes
\begin{displaymath}
  \int f(y)^2 \BE
  H_y(\eta)\bH_y(\delta_x)\bH_y(\delta_z)\bH_y(\delta_w)
  \,\lambda(d(y,z,w))<\infty,\quad \lambda\text{-a.e.}\; x.
\end{displaymath}
This follows from \eqref{eq:int-3-bis}, since $\bH_y(\delta_x)\leq 1$.
The terms which appear in \cite[Corollary~7.3]{las:mol:23} 
are easily derived from the corresponding terms in
\cite[Theorem~7.1]{las:mol:23} 
by noticing that
\begin{displaymath}
  H_x(\eta_t)\bH_x(\eta_t+\delta_z)=H_x(\eta_t)\bH_x(\delta_z). 
\end{displaymath}

\subsection*{Proof of Lemma~8.1}

Conditions (H1)--(H4) are easy to check. We now prove that the
equality (8.2) from \cite{las:mol:23} holds.  It suffices to consider
$\mu\ne 0$. We need to check two set inclusions.  Take $x\in[\mu]$,
that is, $\partial(\mu+\delta_x)=\partial\mu$.  By (H2),
$x\notin\partial\mu$. Assume that $x$ does not belong to the
right-hand side of (8.2) from \cite{las:mol:23}, that is,
$x(r)>(\sup\mu)(r)$ for some $r\in\BS$.  By (8.1) from
\cite{las:mol:23}, this means that $x\in \partial(\mu+\delta_x)$.
Hence $x\in\partial\mu$, a contradiction.
  
Assume conversely that $x\le \sup\mu$ for some $x\in\BX$ such that
$x\notin \partial\mu$.  We need to show that
$\partial(\mu+\delta_x)=\partial\mu$.  Take $x'\in \partial\mu$. By
definition, $(\sup\mu)(r) >(\sup\mu_{-x'})(r)$ for some
$r\in\BS$. Hence, $(\sup\mu)(r)=x'(r)$. Then there exists an open
neighbourhood $U$ of $r$ such that $(\sup\mu)(s)=x'(s)$ for all
$s\in U$. Indeed, assume this was not true. In view of (F3) only a
finite number of functions might contribute to the supremum in a
neighbourhood of $r$. These functions take values at $r$, which are
strictly smaller than $x'(r)$ and continuity of these functions imply
that they are strictly smaller than $x'(s)$ for $s$ from a (possibly
different) neighbourhood of $r$ denoted by $U$.
  
Since $x\le \sup\mu$, we have $\sup(\mu+\delta_{x})(s)=x'(s)$ for all
$s\in U$.  For the sake of a contradiction, assume now that
$x'\notin \partial(\mu+\delta_{x})$. Then
\begin{align*}
  x'(s)=(\sup\mu)(s)=(\sup(\mu+\delta_x))(s)
  =(\sup(\mu_{-x'}+\delta_x))(s)=x(s)
\end{align*}
for all $s\in U$, implying that $x=x'$ on $U$. By (F2), $x=x'$.  Since
$x\notin\partial\mu$, this is a contradiction.  Hence, we have shown
that $\partial\mu\subset \partial(\mu+\delta_{x})$.
  
Now assume that $x'\in\partial(\mu+\delta_x)$ and
$x'\notin\partial\mu$. Then $\sup\mu=\sup\mu_{-x'}$ and there exists
an $r\in\BS$ such that
\begin{displaymath}
  \sup((\mu+\delta_x)_{-x'}(r)<(\sup(\mu+\delta_x))(r).
\end{displaymath}
The left-hand side equals $\max((\sup\mu_{-x'})(r),x(r))$ and the
right-hand side takes the same value, which is a contradiction.

To prove that $\partial$ is measurable, it suffices to show that
$(x,\mu)\mapsto\I\{\sup \mu\ne \sup\mu_{-x}\}$ is measurable on
$\BX\times\bN$. Let $D$ be a countable dense subset of $\BS$.  By
the continuity property of the functions in $\BX$ we have that
$\sup \mu\ne \sup\mu_{-x}$ if and only if there exists a $y\in \mu$
such that $y(r)>x(r)$ for some $r\in D$. Hence it suffices to prove
for any fixed $r\in D$ that
$(x,\mu)\mapsto \int \I\{y(r)>x(r)\}\,\mu(dy)$ is measurable.  But
this follows from the assumed measurability of
$(x,y)\mapsto(x(r),y(r))$.

\subsection*{Proof of Lemma~8.2}

Let $(K_n)_{n\in\N}$ be an increasing sequence of compact subsets of
$\BS$, eventually covering any given compact set.  Let $a_n>0$,
$n\in\N$, be a decreasing sequence such that $a_n\to0$ as
$n\to\infty$.  In view of (F3) we can choose sets $B_n$ in
Proposition~3.3 from \cite{las:mol:23} as
\begin{align*}
  B_n:=\big\{x\in \BX:\text{$x(r)\ge \varpi_{-n,a_n}(r)$
  for some $r\in K_n$}\}\cup\{\varpi \big\}
\end{align*}
if $\varpi\in\BX$, otherwise we omit it in $B_n$.  We shall show
that
\begin{align}
  \label{eq:D1a}
  \lim_{n\to\infty} \I \big\{x\in \partial \mu_{B_n}\big\}
  &=\I \big\{x\in \partial\mu \big\},\quad (\mu,x)\in\bN\times\BX,\\
  \label{eq:D3a}
  \lim_{n\to\infty} \I \big\{x\in [\psi_{B_n}]\big\}&=\I \big\{x\in [\psi]\big\},
                                                      \quad \Pi_\lambda\otimes\lambda
      \text{-a.e.\ $(\psi,x)\in\bN\times\BX$}.
\end{align}
Let $x\in\partial\mu$. Hence, there exists an $s\in\BS$ such that
$x(s)>(\sup \mu_{-x})(s)$. Choose $n$ so large that $s\in K_n$ and
$x(s)>a_n+\varpi(s)$. Then $x\in B_n$ and
$x(s)>(\sup (\mu_{B_n})_{-x})(s)$. Hence, $x\in\partial\mu_{B_n}$.
Assume, conversely, that $x\in\partial\mu_{B_n}$ for some $n\in\N$.
Then $x(s)>\sup (\mu_{B_n})_{-x}(s)$ for some $s\in K_n$. In
particular, $x(s)>a_n+\varpi(s)$.  Since, by definition of $B_n$,
$(\sup \mu_{\BX\setminus B_n})(s)\le a_n+\varpi(s)$, we obtain
$x(s)>(\sup \mu_{-x})(s)$ and hence $x\in\partial\mu$. Thus,
\eqref{eq:D1a} holds.

Next, we prove \eqref{eq:D3a}.  Since $\psi_{B_n}\leq\psi$, (H4)
yields that $[\psi_{B_n}]\subset[\psi]$.  Assume that $x\in[\psi]$ for
a pair $(\psi,x)$ which satisfies (F4).  By Lemma~8.1 from
\cite{las:mol:23}, $x\le \sup\psi$ and $x\notin\partial\psi$.  By (F4)
there exists a finite $\psi'\le \psi$ such that $x\le\sup\psi'$. Since
$B_n\uparrow \BX$, we have $x\le\sup\psi_{B_n}$ for all sufficiently
large $n$. Moreover, by \eqref{eq:D1a} we have
$x\notin\partial\psi_{B_n}$ for all sufficiently large $n$. Lemma~8.1
from \cite{las:mol:23} implies that $x\in[\psi_{B_n}]$ for all
sufficiently large $n$.

Property \eqref{eq:D3a} yields that 
\begin{displaymath}
  \lim_{n\to\infty} \I \big\{x\in \mu_{[\psi_n]}\big\}
  =\I \big\{x\in \mu_{[\psi]}\big\},
  \quad \mu\in\bN,\; \Pi_\lambda\otimes\lambda
  \text{-a.e.\ $(\psi,x)\in\bN\times\BX$}. 
\end{displaymath}
Noticing that $\mu_B$ is finite for $B=B_m$, we have
\begin{equation}
  \label{eq:11}
  \I \big\{\mu_{[\eta]\cap B}= \mu_{[\eta_{B_n}]\cap B}\big\}
  \leq \sum_{x\in\mu_B} \big|\I \big\{x\in \mu_{[\eta]}\big\}
  -\I \big\{x\in \mu_{[\eta_{B_n}]}\big\}\big|\to 0\quad \text{as}\;
  n\to\infty
\end{equation}
for all $\mu\in\bN$, so that \eqref{eH61} holds. Furthermore,
$(\partial\mu)_B$ is also finite, since it is dominated by $\mu_B$,
and so
\begin{displaymath}
  \I \big\{(\partial\mu_{B_n})_B=(\partial\mu)_B \big\}\to 1
  \quad\text{as}\; n\to\infty
\end{displaymath}
for all $\mu\in\bN$.  Together with \eqref{eq:11} applied with $\mu$
replaced by $\eta$ and using the dominated convergence theorem, we
obtain \eqref{eH6}. Therefore, conditions of Proposition~3.3 from
\cite{las:mol:23} are satisfied, and so the strong Markov property
holds.

\subsection*{Proof of Lemma~8.10} 

First, note that \cite[Eq.~(8.15)]{las:mol:23}
holds for $u=0$, since
$H_{s,\phi(s)}(\eta)=1$ a.s. Assume that $u>0$. Let
\begin{displaymath}
  A_{s,u}:=\big\{(q,w)\in\R^d\times\R:u\le w-R\|q-s\|^\beta \big\}.
\end{displaymath}
Then 
\begin{align*}
  \BE H_{s,u}(\eta)=\BP(\eta(A_{s,u})=0)
  =\exp\bigg[-\iint\I \big\{u\le w-R\|s-q\|^\beta,w\le \phi(q) \big\}
  \,dw\,dq\bigg].
\end{align*}
Changing variables yields
\begin{displaymath}
  -\log \BE H_{s,\phi(s)-u}(\eta)= 
  \iint\I \big\{\phi(s)-u\le w-R\|s-q\|^\beta,
  w\le \phi(q) \big\}\,dw\,dq. 
\end{displaymath}
Since $\phi(s)\le\phi(q)+R'\|q-s\|^\beta\le\phi(q)+R\|q-s\|^\beta$, 
\begin{align*}
  -\log \BE H_{s,\phi(s)-u}(\eta)
  &\geq \iint\I\big\{\phi(q)+2R\|s-q\|^\beta-u\le w,
    w\le \phi(q)\big\}\,dw\,dq\\
  &=\iint\I\big\{u-2R\|s-q\|^\beta\ge w \ge 0\big\}\,dw\,dq,
\end{align*}
where in the last integral $w$ has been changed to $\phi(q)-w$. Hence,
\begin{align*}
  -\log \BE H_{s,\phi(s)-u}(\eta) \geq \int \big(u-2R\|q\|^\beta\big)_+\,dq
  &=d\kappa_d\int^\infty_0 \big(u-2R r^\beta\big)_+r^{d-1}\,dr\\
  &=\frac{d\kappa_d}{\beta}\left(\frac{u}{2R}\right)^{d/\beta}u \mathrm{B}(2,d/\beta),
\end{align*}
where $\mathrm{B}$ is the Beta-function.  Hence,
\begin{displaymath}
  -\log \BE H_{s,\phi(s)-u}(\eta) \geq
  \frac{\beta\kappa_d}{d+\beta}(2R)^{-d/\beta} u^{(d+\beta)/\beta}. 
\end{displaymath}
To derive the upper bound, we use the inequality
$\phi(s)\ge\phi(q)-R'\|q-s\|^\beta$ to find that
\begin{align*}
  -\log \BE H_{s,\phi(s)-u}(\eta)\leq 
  \iint\I\big\{\phi(q)+(R-R')\|q\|^\beta-u\le w,
  w\le \phi(q)\big\}\,dw\,dq.
\end{align*}
We can then proceed as before.

For the expectation of the product which appears in
\cite[Eq.~(8.16)]{las:mol:23},
we have
\begin{multline*}
  -\log \BE \big[H_{s,\phi(s)-u}(\eta)H_{r,\phi(r)-v}(\eta)\big]
  =-\log \BP\big(\eta(A_{s,\phi(s)-u}\cup A_{r,\phi(r)-v})=0\big)\\
  =\iint\I\big\{\text{$\phi(s)-u\le w-R\|q-s\|^\beta$
    or $\phi(r)-v\le w-R\|q-r\|^\beta$}\big\}
  \I\big\{w\le \phi(z)\big\}\,dw\,dq.
\end{multline*}
Using the inequalities $\phi(s)\le\phi(q)+R\|q-s\|^\beta$ and
$\phi(r)\le\phi(q)+R\|q-r\|^\beta$, we obtain that the right-hand side
is bounded from below by
\begin{align*}
  \int\int^\infty_0\I\big\{
  &\text{$w\le u-2R\|q-s\|^\beta$
    or $w\le v-2R\|q-r\|^\beta$}\big\}\,dw\,dq\\
  &=\int \big(u-2R\|q-s\|^\beta\big)_+\vee \big(v-2R\|q-r\|^\beta\big)_+\,dq\\
  &\geq \int \big(u-2R\|q\|^\beta\big)_+\,dz\vee \int \big(v-2R\|q\|^\beta\big)_+\,dq\\
  &= \frac{\beta\kappa_d}{d+\beta}(2R)^{-d/\beta} (u\vee
    v)^{(d+\beta)/\beta}. \qedhere
\end{align*} 

\subsection*{Proof of Theorem~8.14} 
\label{sec:proof-theorem}

Denote by $c$ (possibly with indices) positive
constants whose values may change from line to line.  Recall that
$\sigma_t^2:=\BV\hat{F}_t$.  By \cite[Corollary~8.13]{las:mol:23},
$\sigma_t^2\geq c_1 t^{1-\gamma\beta/(d+\beta)}$ for all $t\geq1$.
Our aim is to apply \cite[Corollary~7.3]{las:mol:23} 
and confirm that the order of each of the involved terms is 
the same as the order of $\sigma_t^{-1}$.

The functions $h_i$, $i=0,1,2$, from \cite[Eq.~(7.5)]{las:mol:23}
can be calculated as follows
\begin{align*}
  h_i(r,v)
  &=\int \big|f(s,u)\big|^i \bH_{r,v}(\delta_{s,u}) \I\big\{u\le\phi(s)\big\}\,d(s,u)\\
  &=\int \big|f(s,u)\big|^i \I\{g_{r,v}\leq g_{s,u}\} \I\big\{u\le\phi(s)\big\}\,d(s,u)\\
  &=\int \big|f(s,\phi(s)-u)\big|^i \I\big\{v\le
    \phi(s)-u-R\|s-r\|^\beta\big\}\I\{u\geq 0\}\,d(s,u),
\end{align*}
where we changed the variable $u$ to $\phi(s)-u$ and used
\cite[Lemma~8.9]{las:mol:23}.
The H\"older property of $\phi$ yields that
\begin{align*}
  h_i(r,\phi(r)-v)
  &= \int \big|f(s,\phi(s)-u)\big|^i
  \I\big\{\phi(r)-v\leq \phi(s)-u-R\|s-r\|^\beta\big\}\I\{u\geq 0\}\,d(s,u)\\
  &\leq \int \big|f(s,\phi(s)-u)\big|^i
  \I\big\{(R-R')\|s-r\|^\beta\leq v-u\big\}\I\{0\leq u\leq v\}\,d(s,u).
\end{align*}
In particular,
\begin{displaymath}
  h_0\big(r,\phi(r)-v\big)\leq \frac{\kappa_d\beta}{d+\beta}
  v^{(d+\beta)/\beta}(R-R')^{-d/\beta}.
\end{displaymath}
By \cite[Lemma~8.10]{las:mol:23},
\begin{align*}
  \int f(s,u)^2 & \BE H_{s,u}(\eta_t)h_0(s,u)^2\I\{u\leq\phi(s)\}\,d(s,u)\\
  &=\int f\big(s,\phi(s)-u\big)^2 \BE H_{s,\phi(s)-u}(\eta_t)h_0\big(s,\phi(s)-u\big)^2 \I\{u\geq 0\} \,d(s,u)\\
  &\leq c_1 \int f\big(s,\phi(s)-u\big)^2
  \exp\Big[-c_2u^{(d+\beta)/\beta}\Big]u^{2(d+\beta)/\beta}\I\{u\geq 0\}\,d(s,u)\\
  &\leq c_1 \int_0^\infty f_2(u)
  \exp\Big[-c_2u^{(d+\beta)/\beta}\Big]u^{2(d+\beta)/\beta}\,du.
\end{align*}
The last expression is finite, since $f_2$ grows at most
polynomially. Hence, the integrability condition
\cite[Eq.~(7.6)]{las:mol:23}
is satisfied.

Now we bound the terms from \cite[Corollary~7.3]{las:mol:23}.
First,
\begin{align*}
  T_3(t)
  &=t\sigma_t^{-3}\int \big|f(s,u)\big|^3\BE H_{s,u}(\eta_t)\I\big\{u\le\phi(s)\big\}\,d(s,u)\\
  &\le t\sigma_t^{-3}\int^\infty_0 f_3(u)
    \exp\big[-c_2tu^{(d+\beta)/\beta}\big]\,du,
\end{align*}
where we changed the variable $u$ to $\phi(s)-u$ and used
\cite[Eq.~(8.15)]{las:mol:23}.
By \cite[Lemma~8.11]{las:mol:23}, 
\begin{displaymath}
  T_3(t) \leq c t\, t^{-3(1-\gamma\beta/(d+\beta))/2}
  t^{-\gamma\beta/(d+\beta)}=ct^{-1/2+\gamma\beta/2(d+\beta)}, \quad t\geq1.
\end{displaymath}
Next, 
\begin{align*}
  T_4(t)
  &=t^2\sigma_t^{-3}\int\Big(2 h_2(r,v)\big|f(r,v)\big|+3 h_1(r,v)f(r,v)^2\Big)
  \BE H_{r,v}(\eta_t)\I\big\{v\leq \phi(r)\big\}\, d(r,v)\\
  &= t^2\sigma_t^{-3}\int  \Big(2 h_2(r,\phi(r)-v)\big|f(r,\phi(r)-v)\big|+3
    h_1(r,\phi(r)-v)f(r,\phi(r)-v)^2\Big)\\
  &\qquad\qquad \qquad\qquad \qquad\qquad
    \times\BE H_{r,\phi(r)-v}(\eta_t) \I\{0\leq v\}\, d(r,v)\\ 
  &\leq c_3 t^2\sigma_t^{-3} \int
    \Big(2 h_2(r,\phi(r)-v)\big|f(r,\phi(r)-v)\big|+3
    h_1(r,\phi(r)-v)f(r,\phi(r)-v)^2\Big)\\
  &\qquad\qquad \qquad\qquad \qquad\qquad
    \times \exp\big[-c_2tv^{(d+\beta)/\beta}\big]\I\{0\leq v\}\, d(r,v).    
\end{align*}
By inserting the expression for $h_i$ and using the Lipschitz property
of $\varphi$, we obtain for $i,j=1,2$
\begin{multline*}
  A_{ij}(v):=\int h_i(r,\phi(r)-v)\big|f(r,\phi(r)-v)\big|^j\, dr\\
  \leq \iint \big|f(s,\phi(s)-u)\big|^i\big|f(r,\phi(r)-v)\big|^j
  \I\big\{(R-R')\|s-r\|^\beta\leq v-u\big\}\I\{0\leq u\leq v\}\, d(s,u)dr.
\end{multline*}
Changing variable $q=s-r$, applying the Cauchy--Schwarz inequality
to the integral over $r$, and using the definition of the functions
$f_2$ and $f_4$ yield that
\begin{displaymath}
  A_{ij}(v)\leq c \int_0^v \sqrt{f_{2i}(u)}\sqrt{f_{2j}(v)}
  (v-u)^{d/\beta}\, du.
\end{displaymath}
By \cite[Eq.~(8.19)]{las:mol:23}
and \cite[Lemma~8.11]{las:mol:23}, 
\begin{displaymath}
  A_{ij}(v)\asymp v^{d/\beta+\gamma} \quad\text{as}\; v\to 0, 
\end{displaymath}
and $A_{ij}$ grows at most polynomially. Then,
\begin{displaymath}
  T_4(t)\leq c t^2\sigma_t^{-3}
  \int_0^\infty \big(2A_{21}(v)+3A_{12}(v)\big)
  \exp\big[-c_2tv^{(d+\beta)/\beta}\big]\,dv.
\end{displaymath}
By \cite[Lemma~8.11]{las:mol:23},
the integral has the order $t$ to the power
$-(d/\beta+\gamma+1)\beta/(d+\beta)$. Taking into account the lower
bound on $\sigma_t$, we obtain that
\begin{displaymath}
  T_4(t)\leq c t^2 t^{-3(1-\gamma\beta/(d+\beta))/2}
  t^{-(d/\beta+\gamma+1)\beta/(d+\beta)}
  =ct^{-(1-\gamma\beta/(d+\beta))/2},\quad t\geq1.
\end{displaymath}

For $T_5(t)$ with $x:=(s,u)$, $y:=(r,v)$ and $z:=(q,w)$, we obtain
that
\begin{align*}
  T_5(t)
  &=8t^3\sigma_t^{-3}\iint \big|f(s,u)f(r,v) f(q,w)\big|\BE H_{q,w}(\eta_t)
    \I\big\{u\le\phi(s),v\le\phi(r),w\le\phi(q)\big\}\\
  &\qquad\qquad\times \I\big\{w\le u-R\|s-q\|^\beta,w\le v-R\|r-q\|^\beta\big\}
    \,d(u,v,w)\,d(s,r,q)\\
  &\le ct^3\sigma_t^{-3}\int_{\R_+^3}\int \big|f(s,\phi(s)-u)f(r,\phi(r)-v) f(q,\phi(q)-w)\big|
    \exp\big[-c_2t w^{(d+\beta)/\beta}\big]\\
  &\qquad\qquad\times \I\big\{c\|s-q\|^\beta\le w-u,c\|r-q\|^\beta\le w-v\big\}
    \,d(s,r,q)\,d(u,v,w),
\end{align*}
where the first equality follows from \cite[Lemma~8.9]{las:mol:23}.
Changing the variables $r$ and $s$ and using the H\"older inequality
for the integration with respect to $q$, we obtain that
\begin{multline*}
  T_5(t)\le ct^3\sigma_t^{-3}\int_{\R_+^3}\int f_3^{1/3}(u) f_3^{1/3}(v) f_3^{1/3}(w)
  \exp\big[-c_2t w^{(d+\beta)/\beta}\big]\\
  \times \I\big\{c\|s\|^\beta\le w-u,c\|r\|^\beta\le w-v\big\}\,d(s,r)\,d(u,v,w).
\end{multline*}
Therefore,
\begin{displaymath}
  T_5(t)\le c_1t^3\sigma_t^{-3}\int_0^\infty f_3^{1/3}(w)\exp\big[-c_2t
  w^{(d+\beta)/\beta}\big]  g(w)^2\,dw, 
\end{displaymath}
where 
\begin{displaymath}
  g(w):=\int_0^w f_3^{1/3}(v) (w-v)^{d/\beta}  \, dv.
\end{displaymath}
By \cite[Lemma~8.11]{las:mol:23},
$g$ grows at most polynomially and 
$g(w)\asymp w^{(\gamma+2)/3+d/\beta}$ as $w\to0$, that is,
$g(w)^2\asymp w^{2(\gamma-1)/3+2d/\beta+2}$.
By \cite[Lemma~8.11]{las:mol:23},
\begin{displaymath}
  T_5(t)\leq ct^3\sigma_t^{-3} t^{-(2d/\beta+\gamma+2)\beta/(d+\beta)}
  \leq c_1 t^\alpha, \quad t\geq1,
\end{displaymath}
where
\begin{displaymath}
  \alpha=3-\frac{3}{2}\Big(1-\frac{\gamma\beta}{d+\beta}\Big)
  -\Big(\frac{2d}{\beta}+\gamma+2\Big)\frac{\beta}{d+\beta}
  =-\frac{1}{2}+\frac{\gamma\beta}{2(d+\beta)}.
\end{displaymath}

Now consider the term $T_1(t)$ with $x:=(s,u)$ and $y:=(r,v)$:
\begin{align*}
  T_1(t)^2
  &=t^3 \sigma_t^{-4}\iint\I\big\{u\le\phi(s),v\le\phi(r),w\le\phi(q)\big\} 
    f(s,u)^2f(r,v)^2\BE H_{s,u}(\eta_t)H_{r,v}(\eta_t)\\
  &\qquad\qquad \times \I\big\{u\le w-R\|q-s\|^\beta,v\le w-R\|q-r\|^\beta\big\}
    \,d(u,v,w)\,d(s,r,q)\\
  &\le ct^3 \sigma_t^{-4}\iint f(s,\phi(s)-u)^2f(r,\phi(r)-v)^2 
    \exp\big[-c_2t(u\vee v)^{(d+\beta)/\beta}\big]\\
  &\qquad\qquad\times \I\big\{(R-R')\|q-s\|^\beta\le u-w,
    (R-R')\|q-r\|^\beta\le v-w\big\}\,d(s,r,q)\,d(u,v,w),
\end{align*}
where we used \cite[Eq.~(8.16)]{las:mol:23}.
Changing the variables $s$ and $r$
and applying the Cauchy--Schwarz inequality to the integration with
respect to $q$, we obtain that
\begin{align*}
  T_1(t)^2
  &\leq c t^3 \sigma_t^{-4}\iint \sqrt{f_4(u)}\sqrt{f_4(v)} 
  \exp\big[-c_2t(u\vee v)^{(d+\beta)/\beta}\big]\\
  &\qquad\qquad \qquad\qquad
    \times\I\big\{(R-R')\|s\|^\beta\le u-w,(R-R')\|r\|^\beta\le
    v-w\big\}\,d(s,r)\,d(u,v,w)\\
  &\le ct^3 \sigma_t^{-4}\int \I\{u\ge w\ge 0,v\ge w\}
  \sqrt{f_4(u)}\sqrt{f_4(v)} \\
  &\qquad\qquad \qquad\qquad\times\exp\big[-c_2t(u\vee v)^{(d+\beta)/\beta}\big]
  (u-w)^{d/\beta}(v-w)^{d/\beta}\,d(u,v,w).
\end{align*}
We split the integral into the part when $u\vee v\leq \eps$ and the
rest, where $\eps>0$ is chosen to ensure that $f_4(u)\leq c u^{\gamma-1}$ for
all $u\in(0,\eps)$ and use the symmetry of integrals with respect to
$u$ and $v$. Then
\begin{displaymath}
  T_1(t)\leq c\sigma_t^{-2} (T'_1(t)+T''_1(t))^{1/2},
\end{displaymath}
where 
\begin{align*}
  T'_1(t)&:=c_1 t^3 \int \I\{u\ge w\ge 0, v\ge w\}
  u^{(\gamma-1)/2}v^{(\gamma-1)/2}\\
  &\qquad\qquad \qquad\qquad \times
  \exp\big[-c_2t(u\vee v)^{(d+\beta)/\beta}\big]
  (u-w)^{d/\beta}(v-w)^{d/\beta}\,d(u,v,w),\\
  T''_1(t)&:=2c t^3 \int_\eps^\infty \int_0^u u^mv^m  
  \exp\big[-c_2t u^{(d+\beta)/\beta}\big]
  \Big(\int_0^u (u-w)^{d/\beta}(v-w)^{d/\beta}\,dw\Big)\,dv du,
\end{align*}
where $m$ is the power in the polynomial upper bound on $f_4$.  
Replacing $(u,v,w)$ by $t^{-\beta/(d+\beta)}(u,v,w)$, we see that
\begin{align*}
  T'_1(t)\le cc' t^3
  t^{-(2d/\beta+\gamma+2)\beta/(d+\beta)},
\end{align*}
where  
\begin{align*}
  c':=\int\I\{u\ge w\ge 0,v\ge w\} \exp\big[-c_2(u\vee v)^{(d+\beta)/\beta}\big]
  (u-w)^{d/\beta}(v-w)^{d/\beta}\,d(u,v,w)
\end{align*}
is easily seen to be finite. Therefore,
\begin{align*}
  \sigma_t^{-2}(T'_1(t))^{1/2} \le c t^\alpha
\end{align*}
with 
\begin{displaymath}
  \alpha=-1+\frac{\gamma\beta}{d+\beta}+\frac{3}{2}
  -\frac{1}{2}\Big(\frac{2d}{\beta}+\gamma+2\Big)\frac{\beta}{d+\beta}
  =-\frac{1}{2} +\frac{\gamma\beta}{2(d+\beta)}.
\end{displaymath}
Further,
\begin{align*}
  T''_1(t)
  &\leq c_1 t^3 \int_\eps^\infty u^m \int_0^u v^m  
  \exp\big[-c_2t u^{(d+\beta)/\beta}\big]
    \Big(\int_0^u (u-w)^{2d/\beta}\,dw\Big)\,dv du\\
  &\leq c_3 t^3 \int_\eps^\infty u^m u^{m+1} u^{2d/\beta +1}
    \exp\big[-c_2t u^{(d+\beta)/\beta}\big]\, du.
\end{align*}
This can easily be bounded by $ct^{\delta}e^{-\delta' t}$
for some $\delta>0$ and $\delta'=c_2/2$.

\end{document}